\documentclass[12pt]{amsart}
\usepackage{amsmath,mathrsfs,amsthm,latexsym,amsbsy,amssymb, setspace,nicefrac, yhmath, amscd,eucal}
\usepackage[active]{srcltx}
\usepackage{amsrefs, mathrsfs}
\usepackage[english]{babel}

\newtheorem{theorem}{Theorem}[section]
\newtheorem{remark}[theorem]{Remark}
\newtheorem{corollary}[theorem]{Corollary}
\newtheorem{lemma}[theorem]{Lemma}
\newtheorem{prop}[theorem]{Proposition}
\theoremstyle{definition}\newtheorem{df}[theorem]{Definition}
\theoremstyle{definition}
\numberwithin{equation}{section}

\usepackage{enumerate}
\def\minus {\backslash }
\usepackage{multicol,amstext,graphics,graphicx}
\usepackage{fancyhdr}
\usepackage{hyperref}

\allowdisplaybreaks

\newcommand{\dpcm}{ \hfill $\square$}
\newcommand{\cM}{{\mathcal M}}
\newcommand{\cF}{{\mathcal F}}

\newcommand{\supp}{{\rm supp}}
\newcommand{\C}{{\mathbb C}}
\newcommand{\bF}{{\mathbb F}}
\newcommand{\R}{{\mathbb R}}
\newcommand{\N}{{\mathbb N}}

\newcommand{\Opt}{{\rm Opt}}

\title[Duality and quotients spaces of generalized Wasserstein spaces]{Duality and quotient spaces of generalized Wasserstein spaces}
\author{Nhan-Phu Chung}
\address{Nhan-Phu Chung, Department of Mathematics, Sungkyunkwan University, Suwon 440-746, Korea. Tel: +82 031-299-4819} 
\email{phuchung@skku.edu;phuchung82@gmail.com}
\author{ Thanh-Son Trinh}
\address{Thanh-Son Trinh, Department of Mathematics, Sungkyunkwan University, Suwon 440-746, Korea.}
\email{sontrinh@skku.edu}
\begin{document}
\date{\today}
\maketitle

\begin{abstract}
In this article, using ideas of Liero,  Mielke and Savar\'{e} in \cite{Liero} we establish a Kantorovich duality for generalized Wasserstein distances $W_1^{a,b}$ on a generalized Polish metric space, introduced by Picolli and Rossi in \cite{PR14}. As a consequence, we give another proof that $W_1^{a,b}$ coincide with flat metrics which is a main result of \cite{PR16}, and therefore we get a result of independent interest that $\left(\mathcal{M}(X), W^{a,b}_1\right)$ is a geodesic space for every Polish metric space $X$. We also prove that $(\cM^G(X),W_p^{a,b})$ is isometric isomorphism to $(\cM(X/G),W_p^{a,b})$ for isometric actions of a compact group $G$ on a Polish metric space $X$; and several results of Gromov-Hausdorrf convergence and equivariant Gromov-Hausdorff convergence of generalized Wasserstein spaces. The latter results were proved for standard Wasserstein spaces in \cite{LV},\cite{Garcia} and \cite{NPC} respectively.      
\end{abstract}

\section{Introduction}  
The Monge-Kantorovich's balanced optimal transport problem has been studied extensively after pioneer works of Kantorovich on 1940s \cite{Kant42,Kant48}. In connection with this problem, Wassertein distances in the space of probability measures are powerful tools to study gradient flows and partial differential equations \cite{Ambrosio} and theory of Ricci curvature bounded below for general metric-measure spaces \cite{LV,Sturm}.

Recently, unbalanced optimal transport problems and various generalized Wasserstein distances on the space of finite measures have been introduced and investigated by numerous authors \cite{CPSV,KMV,Liero, PR14}. In \cite{PR14}, Piccoli and Rossi defined a generalized Wassertein distance $W^{a,b}_p(\mu,\nu)$, combining the usual Wasserstein distance and $L^1$-distance. After that, they also proved the generalized Benamou-Breiner formula for $W^{a,b}_p$ and showed that the generalized Wasserstein distance $W_1^{1,1}$ coincides with the flat metric \cite{PR16}. As natural we would ask which other properties of standard Wasserstein distances still hold for generalized Wasserstein distances $W^{a,b}_p$.

In this article, our first result is the Kantorovich duality for the distance $W^{a,b}_1$. In \cite{Liero}, Liero, Mielke and Savar\'{e} established Kantorovich duality for various Entropy-Transport problems where entropy functions satisfy coercive conditions. As our nonsmooth entropy function $F(s)=a|1-s|$ is not superlinear and the cost function $b.d(\cdot,\cdot)$ does not have compact sublevels when $X$ is a general Polish metric space we can not get the Kantorovich duality in our setting directly from \cite{Liero}. However, inspiring from their methods we can prove that   
\begin{theorem}
\label{T-duality of generalized Wasserstein spaces}
Let $X$ be a Polish metric space. For any $\mu_1,\mu_2\in\mathcal{M}(X)$, we have 
\begin{align*}
W_1^{a,b}(\mu_1,\mu_2)=\sup\limits_{(\varphi_1,\varphi_2)\in\Phi_W} \sum\limits_i \int_X I\left(\varphi_i(x)\right) d\mu_i(x),
\end{align*}
where $I(\varphi)=\inf\limits_{s\geq 0} \left(s\varphi+a\vert 1-s\vert\right) \text{ for }\varphi\in\mathbb{R}$, and 
$$\Phi_W:= \{(\varphi_1,\varphi_2)\in C_b(X)\times C_b(X)\;\vert\; \varphi_1(x)+\varphi_2 (y)\leq b.d(x,y)\text{ and } \varphi_1 (x), \varphi _2(y)\geq -a,\;\forall x,y\in X\}.$$
\end{theorem}
As a consequence, we get a version of Kantorovich-Rubinstein theorem for generalized Wasserstein distance $W_1^{a,b}$, which is a main result of \cite{PR16} and is proved by a different method there. 
\begin{theorem}
\label{T-flat metrics}
Let $(X,d)$ be a Polish metric space, Then for every $a,b>0,  \mu,\nu\in \cM(X)$ we have 
$$W_1^{a,b}(\mu,\nu)=\sup\big\{\int_X fd(\mu-\nu):f\in \bF\big\},$$
where $\bF:=\big\{f\in C_b(X), \|f\|_\infty\leq a, \|f\|_{Lip}\leq b\big\}$.
\end{theorem}
And from that we get the following result which is independent of interest.
\begin{corollary}
\label{C-geodesic space}
Let $(X,d)$ be a Polish metric space and let $a,b>0$. Then $\left(\mathcal{M}(X), W^{a,b}_1\right)$ is a geodesic space.
\end{corollary}
On the other hand, in \cite{LV} Lott and Villani established an isometric isomorphism for the Wasserstein spaces $P_2^G(X)$ of $G$-invariant elements in $P_2(X)$ and $P_2(X/G)$, where $X/G$ is the quotient space of $X$ induced from an isometric action of a compact group $G$ on a compact metric space $X$. Later, this result is extended for general metric spaces $X$ in \cite{Garcia}. Our second result is its version for generalized Wasserstein distances $W_p^{a,b}$.
\begin{theorem}\label{T-isometry}
Let a compact group $G$ act on the right of a locally compact Polish metric space $(X,d)$ by isometries. Let $p:X\to X/G$ be the natural quotient map and numbers $a,b>0$, $p\geq 1$. Then 
\begin{enumerate}
\item the map $p_\sharp:\cM_p(X)\to \cM_p(X/G)$ is onto and furthermore for every $\nu^*\in \cM_p(X/G)$ we can find $\mu\in \cM_p^G(X)$ such that $p_\sharp \mu=\nu^*$;
\item $W_p^{a,b}(p_\sharp \mu,p_\sharp\nu)\leq W_p^{a,b}(\mu,\nu)$ for every $\mu,\nu\in \cM(X)$;
\item the map $p_\sharp:(\cM^G(X),W_p^{a,b})\to (\cM(X/G),W_p^{a,b})$ is an isometry;
\item the map $p_\sharp:(\cM^G_p(X),W_p^{a,b})\to (\cM_p(X/G),W_p^{a,b})$ is an isometry. 
\end{enumerate}
\end{theorem}
Lastly, we prove Gromov-Hausdorff convergence of the generalized Wasserstein spaces and equivariant Gromov-Hausdorff convergence for induced actions on generalized Wasserstein spaces. These results have been established for standard Wasserstein spaces in \cite{LV} and \cite{NPC} respectively.
\begin{theorem}\label{T-convergence of generalized Wasserstein spaces} Let $\left\{\left(X_n, d_n\right)\right\}$ be a sequence of bounded, Polish metric spaces and $C>0$. If $\left\{\left(X_n, d_n\right)\right\}$ converges in the Gromov-Hausdorff topology to a bounded, Polish metric space $(X,d)$ then $\left\{\left(\cM_p^C\left(X_n\right),W^{a,b}_p\right)\right\}$ converges in the Gromov-Hausdorff topology to $\left(\cM_p^C(X), W^{a,b}_p\right)$ for every $a,b>0,p\geq 1$.
\end{theorem}
\begin{theorem}\label{T-actions on generalized Wasserstein spaces}
Let $\{\alpha_n\}$ be a sequence of continuous actions of a topological group $G$ on bounded, Polish metric spaces $\left\{X_n\right\}$ and let $C>0 ,p\geq 1, a,b>0$. Let $\left(\alpha_n\right)_{\sharp}$ be the induced action of $\alpha_n$ on the space $(\cM_p^C\left(X_n\right),W_p^{a,b})$, for every $n\in\mathbb{N}$. If $\lim_{n\rightarrow \infty}d_{mGH}\left(\alpha_n,\alpha\right)=0$ for some continuous action $\alpha$ of $G$ on a bounded, Polish metric space $X$. We denote by $\alpha_{\sharp}$ the induced action of $\alpha$ on the space $(\cM_p^C\left(X_n\right),W_p^{a,b})$. Then $\lim_{n\rightarrow \infty}d_{mGH}\left(\left(\alpha_n\right)_{\sharp},\alpha_{\sharp}\right)=0$.
\end{theorem}
The paper is organized as following. In section 2, we review generalized Wasserstein distances and equivariant Gromov-Hausdorff distance. In section 3, we will prove theorem \ref{T-convergence of generalized Wasserstein spaces} and theorem \ref{T-actions on generalized Wasserstein spaces}. Theorem \ref{T-isometry} will be proved in section 4. Finally, in section 5 we will prove theorem \ref{T-duality of generalized Wasserstein spaces}, theorem \ref{T-flat metrics} and corollary \ref{C-geodesic space}. Furthermore, in this last section we also present some interesting consequences of theorem \ref{T-flat metrics}, and several results of optimal plans and dual optimal for $W_1^{a,b}$. 

\textbf{Acknowledgements:} Part of this paper was carried out when N. P. Chung visited University of Science, Vietnam National University at Hochiminh city on January 2019. He is grateful to Dang Duc Trong for his warm hospitality. The authors were partially supported by the National Research Foundation of Korea (NRF) grants funded by the Korea government (No. NRF- 2016R1A5A1008055 , No. NRF-2016R1D1A1B03931922 and No. NRF-2019R1C1C1007107). We also thank Benedetto Piccoli and Francesco Rossi for pointing \cite{Hanin1} out to us. 
\section{Preliminaries}

\subsection{Notations, Wasserstein spaces and generalized Wasserstein spaces}
\hfill

First, we review notations we use in the paper and recall the definitions of Wasserstein distances and some of their properties. For more details, readers can see \cite{V03,V09}. 

Let $(X,d) $ be a metric space. We denote by $\mathcal{M}(X)$ and $\mathcal{P}(X)$ the sets of all nonnegative Borel measures with finite mass and all probability Borel measures, respectively.

Given a Borel measure $\mu$, we denote its mass by $\vert \mu\vert :=\mu (X)$. A set $M\subset \mathcal{M}(X)$ is bounded if $\sup_{\mu\in M}\vert \mu\vert <\infty$, and it is \textit{tight} if for every $\varepsilon>0$, there exists a compact subset $K_\varepsilon$ of $X$ such that for all $\mu\in M$, we have $\mu\left(X\backslash K_\varepsilon\right)\leq \varepsilon$.

For every $\mu,\nu\in \mathcal{M}(X)$, we say that $\mu$ is absolutely continuous with respect to $\nu$ and write $\mu \ll \nu$ if $\nu(A)=0$ yields $\mu(A)=0$ for every Borel subset $A$ of $X$. We call that $\mu$ and $\nu$ are mutually singular and write $\mu \perp \nu$ if there exists a Borel subset $B$ of $X$ such that $\mu(B)=\nu(X\backslash B)=0$. We write $\mu\leq \nu$ if for all Borel subset $A$ of $X$ we have $\mu(A)\leq \nu(A)$.
\begin{theorem} (Prokhorov's theorem)
Let $(X,d)$ be a metric space. If a subset $M\subset \cM(X)$ is bounded and \textit{tight} then $M$ is relatively compact under the weak*- topology. 
\end{theorem}
 
For every $p\geq 1$, we denote by $\mathcal{M}_p(X)$ (reps. $\mathcal{P}_p (X)$) the space of all measures $\mu\in  \mathcal{M}(X)$ (reps. $\mathcal{P}(X)$) with finite $p$-moment, i.e. there is some (and therefore any) $x_0\in X$ such that $$\int_{X}d^p\left(x,x_0\right)d\mu(x)<\infty.$$

For every measures $\mu,\nu\in \mathcal{M}(X)$, a Borel probability measure $\pi$ on $X\times X$ is called a transference plan between $\mu$ and $\nu$ if $$\vert \mu\vert \pi (A\times X)=\mu(A)\text{ and }\vert \nu\vert \pi (X\times B)=\nu(B),$$
for every Borel subsets $A,B$ of $X$. We denote the set of all transference plan between $\mu$ and $\nu$ by $\Pi(\mu,\nu)$.

 Given measures $\mu,\nu\in\mathcal{M}_p(X)$ with the same mass, i.e. $\vert \mu\vert=\vert \nu\vert$. The Wasserstein distance between $\mu$ and $\nu$ is defined by $$W_p(\mu,\nu):=\left(\vert \mu\vert\inf_{\pi\in \Pi (\mu,\nu)}J_p (\pi)\right)^{1/p},$$
where $J_p(\pi)=\int_{X\times X}d^p(x,y)d\pi(x,y)$. For each $\mu,\nu\in\mathcal{M}(X)$ with $|\mu|=|\nu|$, we denote by $\Opt_p(\mu,\nu)$ the set of all $\pi\in\Pi(\mu,\nu)$ such that $W^p_p(\mu,\nu)=\vert \mu\vert J_p(\pi)$. If $(X,d)$ is a Polish metric space, i.e. $(X,d)$ is complete and separable then $\Opt_p(\mu,\nu)$ is nonempty. This result follows from \cite[Theorem 1.3]{V03} by setting $\mu^*=\mu/\vert \mu\vert,\nu^*=\nu/\vert \nu\vert\in \mathcal{P}_p(X)$.

Let $\left(X,d_X\right)$ and $\left(Y,d_Y\right)$ be metric spaces and $f:X\rightarrow Y$ be a Borel map. We denote by $f_{\sharp}\mu\in\mathcal{M}(Y)$ the push-forward measure defined by $$f_{\sharp}\mu(B):=\mu\left(f^{-1}(B)\right),$$
for every Borel subset $B$ of $Y$.

We now review the definitions of the generalized Wasserstein distances introduced by Piccoli and Rossi in \cite{PR14}. Note that although in \cite{PR14} the authors only presented for the case $X=\R^d$ their methods work for a general Polish metric space $X$. 
\begin{df}
Let $X$ be a Polish metric space and let $a,b>0,p\geq 1$. For every $\mu,\nu\in \mathcal{M}(X)$, the generalized Wasserstein distance $W^{a,b}_p$ between $\mu$ and $\nu$ is defined by \begin{align*}
W^{a,b}_p (\mu,\nu) := \inf\limits_{\begin{matrix}
\widetilde{\mu},\widetilde{\nu}\in\mathcal{M}_p(X) \\ \vert \widetilde{\mu}\vert =\vert \widetilde{\nu}\vert
\end{matrix}} C\left(\widetilde{\mu}, \widetilde{\nu}\right),
\end{align*}
where $ C\left(\widetilde{\mu}, \widetilde{\nu}\right)= a\left\vert \mu-\widetilde{\mu}\right\vert+a\left\vert \nu-\widetilde{\nu}\right\vert+b\,W_p\left(\widetilde{\mu},\widetilde{\nu}\right).$
\end{df}
\begin{prop}
(\cite[Proposition 1]{PR14})
\label{P-less measures for general Wassertein spaces}
If $X$ is a Polish metric space then $\left(\mathcal{M}(X), W^{a,b}_p\right)$ is a metric space. Moreover, there exists $\widetilde{\mu},\widetilde{\nu}\in \mathcal{M}_p(X)$ such that $\vert \widetilde{\mu}\vert = \vert \widetilde{\nu}\vert, \widetilde{\mu}\leq \mu, \widetilde{\nu}\leq \nu$ and $W^{a,b}_p (\mu,\nu) =  C\left(\widetilde{\mu}, \widetilde{\nu}\right)$.
\end{prop}
If measures $\widetilde{\mu}, \widetilde{\nu}\in \mathcal{M}_p(X)$ with the same mass such that $W^{a,b}_p (\mu,\nu) =  C\left(\widetilde{\mu}, \widetilde{\nu}\right)$ then we say that $\left(\widetilde{\mu}, \widetilde{\nu}\right)$ is an optimal for $W^{a,b}_p (\mu,\nu)$.  
\begin{prop}\label{P-completeness}
(\cite[Proposition 4]{PR14}) If $(X,d)$ is a Polish metric space then $\left(\mathcal{M}(X),W^{a,b}_p\right)$ is a complete metric space.
\end{prop}
\subsection{Equivariant Gromov-Hausdorff distances for group actions}
\hfill

First, we recall the definition of Gromov-Hausdorff distance between two metric spaces. For more details, see standard references \cite{Burago,Shioya}.

Let $(X,d)$ be a metric space. For every $\varepsilon>0$, the $\varepsilon$-neighborhood of a subset $S$ of $X$, denoted by $B_\varepsilon (S)$, is defined as $B_\varepsilon (S)=\bigcup_{x\in S}B_\varepsilon (x)$.
\begin{df}
Let $X$ and $Y$ be subsets of a metric space $(Z,d)$. The Hausdorff distance between $X$ and $Y$, denoted by $d_H(X,Y)$ is defined as follow $$d_H(X,Y):=\inf\left\lbrace \varepsilon>0\mid X\subset B_\varepsilon (Y)\text{ and }Y\subset B_\varepsilon (X)\right\rbrace .$$
\end{df}
\begin{df}
Let $X$ and $Y$ be metric spaces. The Gromov-Hausdorff distance $d_{GH}(X,Y)$ is the infimum of $r>0$ such that there exist a metric space $(Z,d)$ and subspaces $X'$ and $Y'$ of $Z$ which are isometric to $X$ and $Y$ respectively such that $d_H(X',Y')<r$.
\end{df}
\begin{df}
Given two bounded metric spaces $\left(X_1, d_1\right), \left(X_2, d_2\right)$. An $\varepsilon$-Gromov-Hausdorff approximation from $X_1$ to $X_2$ is a map $f: X_1\rightarrow X_2$ such that
\begin{enumerate}[(i)]
\item For every $x_1, x_1'\in X_1$ then $\left\vert d_2\left(f(x_1), f(x_1')\right)-d_1\left( x_1, x_1'\right)\right\vert\leq \varepsilon$
\item For every $x_2\in X_2$, there exists $x_1\in X_1$ such that $d_2 \left( f(x_1), x_2\right)\leq \varepsilon$.
\end{enumerate}
\end{df}
If $f$ is an $\varepsilon$-Gromov-Hausdorff approximation from $X_1$ to $X_2$ then it has an approximate inverse $f': X_2\rightarrow X_1$ which is a $3\varepsilon$-Gromov-Hausdorff approximation from $X_2$ to $X_1$. To see this, we will construct $f'$ as follows. Let $x_2\in X_2$, choose $x_1\in X_1$ such that $\mathbf{d}_2\left( x_2, f(x_1)\right)\leq \varepsilon$, since the second condition of definition of $f$. Setting $f'(x_2)=x_1$ then $f'$ is a $3\varepsilon$-Gromov-Hausdorff approximation from $X_2$ to $X_1$. Moreover, it is clear that $$
d_1\left(x_1,f'\left(f\left(x_1\right)\right)\right)\leq 2\varepsilon \text{ and }d_2\left(x_2,f\left(f'\left(x_2\right)\right)\right)\leq\varepsilon \mbox{ for every } x_1\in X_1, x_2\in X_2.$$

$\;\,$ Now we review the equivariant Gromov-Hausdorff distances. They were introduced first by Fukaya in \cite{Fu86,Fu88,Fukaya,FY} for isometric actions. After that they have been studied further for general actions \cite{AM,NPC,DLM,KDD}.

Let $(X,d)$ be a metric space. The $C^0$ distance between the maps $f,g: (X,d)\rightarrow (X,d)$ is defined by $d_{sup}(f,g):= \sup\limits_{x\in X}d(f(x),g(x)).$
\begin{df}
Let $\alpha$ and $\beta$ be continuous actions of G on metric spaces $\left(X,d_X\right)$ and $\left(Y,d_Y\right)$ respectively. A map $f:G\curvearrowright X \rightarrow G\curvearrowright Y$ is an $\varepsilon$-GH approximation from $\alpha$ to $\beta$ if $f: X\rightarrow Y$ is an $\varepsilon$-isometry satisfying that $d_{sup}\left(f\circ \alpha_g,\beta_g\circ f\right)\leq \varepsilon$ for every $g\in G$. If $f$ is measurable we say that $f$ is an $\varepsilon$-measurable GH approximation.
\end{df}
\begin{df}
Let $\alpha$ and $\beta$ be continuous actions of G on metric spaces $\left(X,d_X\right)$ and $\left(Y,d_Y\right)$ respectively. The equivariant GH-distance $d_{GH}$ and $d_{mGH}$ between $\alpha$ and $\beta$ are defined by \begin{align*}
d_{GH}(\alpha,\beta):= \inf\left\{\varepsilon >0:\exists \varepsilon\text{-GH approximations }f:G\curvearrowright X \rightarrow G\curvearrowright Y\right.\\
\left.\text{and }g:G\curvearrowright Y \rightarrow G\curvearrowright X \right\},
\end{align*}
\begin{align*}
d_{mGH}(\alpha,\beta):= \inf\left\{\varepsilon >0:\exists \varepsilon\text{-measurable GH approximations } f:G\curvearrowright X \rightarrow G\curvearrowright Y\right.\\
\left. \text{ and } g:G\curvearrowright Y \rightarrow G\curvearrowright X\right\},
\end{align*}
and is $\infty$ if the infimum does not exist.
\end{df}

\section{Gromov-Hausdorff convergences for generalized Wasserstein spaces}

In this section, we will prove theorem \ref{T-convergence of generalized Wasserstein spaces} and theorem \ref{T-actions on generalized Wasserstein spaces}.

Let $X$ be a Polish metric space and let $C>0$, $p\geq 1$. We denote by $\cM_p^C(X)$ the space of all measures $\mu\in \cM_p(X)$ such that $|\mu|\leq C$. Note that when $X$ is bounded then $\cM_p(X)=\cM(X)$ for every $p\geq 1$.

\begin{lemma}\label{L-convergence of generalized Wasserstein spaces} Let $\left(X_1, d_1\right)$ and $\left(X_2,d_2\right)$ be two bounded, Polish metric spaces and $C>0$. If $f: \left(X_1, d_1\right)\rightarrow \left(X_2, d_2\right)$ is an $\varepsilon$-Gromov-Hausdorff approximation and measurable then $f_{\sharp}: \left(\cM_p^C\left(X_1\right), W_p^{a,b}\right)\rightarrow \left(\cM_p^C\left(X_2\right), W_p^{a,b}\right)$ is an $\widetilde{\varepsilon}$-Gromov-Hausdorff approximation, where
$$\widetilde{\varepsilon} = 8bC^{2/p}\varepsilon +b\left(9pC\left(\mathrm{diam}(X_1)^{p-1} + \mathrm{diam}(X_2)^{p-1}\right)\varepsilon\right)^{1/p}.$$ 
\end{lemma}
\begin{proof} Given $\mu, \nu\in \cM_p^C\left(X_1\right)$ and let $(\widetilde{\mu}, \widetilde{\nu})\in\cM_p(X)\times \cM_p(X)$ be an optimal for $W_p^{a,b} \left(\mu,\nu\right)$ such that $\vert \widetilde{\mu}\vert = \vert \widetilde{\nu}\vert$ and $\widetilde{\mu}\leq \mu, \widetilde{\nu}\leq \nu$.\\
Setting $\overline{\mu}:= f_{\sharp} \widetilde{\mu}$ and $\overline{\nu}:=f_{\sharp}\widetilde{\nu}$ then $\overline{\mu}\leq f_{\sharp}\mu, \overline{\nu}\leq f _{\sharp}\nu$ and $\vert \overline{\mu}\vert = \vert \overline{\nu}\vert$. Therefore
\begin{align*}
W_p^{a,b}\left(f_{\sharp}\mu,f_{\sharp}\nu\right) \leq a\vert f_{\sharp}\mu - \overline{\mu}\vert +a\vert f_{\sharp}\nu - \overline{\nu}\vert +bW_p\left( \overline{\mu},\overline{\nu}\right).
\end{align*}
Let $\pi_1$ be an optimal transference plan between $\widetilde{\mu}$ and $\widetilde{\nu}$. Define $\pi_2 := \left(f\times f\right)_{\sharp}\pi_1$. Then $\pi_2 \in \Pi \left( f_{\sharp}\widetilde{\mu}, f_{\sharp}\widetilde{\nu}\right)$. Therefore, 
\begin{align*}
W_p^p\left(f_{\sharp}\widetilde{\mu}, f_{\sharp}\widetilde{\nu}\right)\leq \vert f_{\sharp}\widetilde{\mu}\vert \displaystyle\int_{X_2\times X_2} d_2^p \left(x_2, y_2\right) d\pi_2(x_2,y_2) = \vert \widetilde{\mu}\vert \displaystyle\int_{X_1\times X_1} d_2^p\left( f(x_1), f(y_1)\right)d\pi_1(x_1,y_1).
\end{align*}
Applying the mean value theorem for the function $t^p, t\geq 0$ we get $$\left\vert x^p-y^p\right\vert \leq p\vert x-y\vert \max\left\{x^{p-1}, y^{p-1}\right\}\leq p \vert x-y\vert \left(x^{p-1}+y^{p-1}\right).$$  
So for all $x_1,y_1\in X_1$, 
\begin{align*}
\vert d_2^p\left(f(x_1),f(y_1)\right) & -d_1^p\left(x_1, y_1\right)\vert  \leq \\ &\leq p\left\vert d_2\left(f\left(x_1\right),f\left(y_1\right)\right)-d_1\left(x_1,y_1\right) \right\vert \left(d_2^{p-1}\left(f\left(x_1\right),f\left(y_1\right)\right)+d_1^{p-1}\left(x_1,y_1\right)\right)\\
& \leq pM\varepsilon ,
\end{align*}
where $M=\text{diam}\;(X_1)^{p-1}+\text{diam}\;(X_2)^{p-1}$. Therefore,
\begin{align*}
W^p_p\left(f_{\sharp}\widetilde{\mu}, f_{\sharp}\widetilde{\nu}\right)& \leq \vert \widetilde{\mu}\vert \displaystyle\int_{X_1\times X_1}d_1^p\left(x_1, y_1\right)d\pi_1(x_1,y_1)+\vert \widetilde{\mu}\vert  pM\varepsilon\\
&= W_p^p\left(\widetilde{\mu},\widetilde{\nu}\right) +\vert \widetilde{\mu}\vert  pM\varepsilon . 
\end{align*}
Since $\vert \widetilde{\mu}\vert \leq \vert \mu\vert \leq C$, one has $
W_p\left(f_{\sharp}\widetilde{\mu}, f_{\sharp}\widetilde{\nu}\right) \leq W_p \left(\widetilde{\mu}, \widetilde{\nu}\right) + (pCM\varepsilon)^{1/p}.$
Moreover, as $\vert f_{\sharp}\mu-f_{\sharp}\widetilde{\mu}\vert +\vert f_{\sharp}\nu-f_{\sharp}\widetilde{\nu}\vert=\vert \mu-\widetilde{\mu}\vert +\vert \nu-\widetilde{\nu}\vert$ we obtain  
\begin{align*}
W_p^{a,b}\left(f_{\sharp}\mu, f_{\sharp}\nu\right) &\leq a\vert \mu-\widetilde{\mu} \vert +a\vert \nu - \widetilde{\nu}\vert +bW_p\left(\widetilde{\mu}, \widetilde{\nu}\right) + b(pCM\varepsilon)^{1/p}\notag\\
& = W_p^{a,b}\left( \mu, \nu\right) + b(pCM\varepsilon)^{1/p}.
\end{align*}
Now, using \cite[Lemma 4.1]{NPC} there exists a measurable function $f'$ which is a $9\varepsilon$-Gromov-Hausdorff approximation from $X_2$ to $X_1$, and $$d_1\left(x_1,\left(f'\circ f\right)\left(x_1\right)\right)\leq 4\varepsilon\;\text{for all}\; x_1\in X_1\quad\mbox{and}\quad d_2\left(x_2,\left(f\circ f'\right)\left(x_2\right)\right)\leq 4\varepsilon\;\text{for all}\; x_2\in X_2.$$
Using the same argument as above we get that 
\begin{align*}
W_p^{a,b} \left(f'_{\sharp}\left(f_{\sharp}\mu\right), f'_{\sharp}\left(f_{\sharp}\nu\right)\right) \leq W_p^{a,b}\left( f_{\sharp}\mu, f_{\sharp}\nu\right) +b(9pCM\varepsilon)^{1/p}.
\end{align*}
Since $\vert \mu\vert =\vert f'_{\sharp}\left(f_{\sharp} \mu\right)\vert$ we have, 
$
W_p^{a,b}\left(\mu, f'_{\sharp}\left(f_{\sharp} \mu\right)\right)\leq bW_p\left(\mu, f'_{\sharp}\left(f_{\sharp} \mu\right)\right).$
By \cite[Lemma 2.8]{NPC} we get that 
\begin{align*}
W_p^p\left(\mu, f'_{\sharp}\left(f_{\sharp} \mu\right)\right)&\leq  \vert \mu\vert \displaystyle\int_{X_1} d_1^p \left(x_1, f'(f(x_1))\right)d\mu(x_1)\\
& \leq \vert \mu\vert ^2 (4\varepsilon)^p \\
& \leq C^2 (4\varepsilon)^p. 
\end{align*}
And thus, $W_p^{a,b}\left(\mu, f'_{\sharp}\left(f_{\sharp} \mu\right)\right)\leq 4b C^{2/p}\varepsilon .$ Similarly, $W_p^{a,b}\left(\nu, f'_{\sharp}\left(f_{\sharp} \nu\right)\right)\leq 4b C^{2/p}\varepsilon$. 

Therefore,
\begin{eqnarray*}
W_p^{a,b} \left(\mu,\nu\right) &\leq& W_p^{a,b}\left(f'_{\sharp}\left(f_{\sharp}\mu\right), f'_{\sharp}\left(f_{\sharp}\nu\right)\right) + W_p^{a,b}\left(\mu, f'_{\sharp}\left(f_{\sharp} \mu\right)\right) + W_p^{a,b} \left( \nu, f'_{\sharp}\left(f_{\sharp} \nu\right)\right)\\
  &\leq& W_p^{a,b}\left( f_{\sharp}\mu,f_{\sharp}\nu\right)+b(9pCM\varepsilon)^{1/p}+8b C^{2/p}\varepsilon .
\end{eqnarray*}
And hence
\begin{align}\label{3.1}
\left\vert  W_p^{a,b}\left( f_{\sharp}\mu,f_{\sharp}\nu\right) - W_p^{a,b}\left(\mu,\nu\right)\right\vert \leq b(9pCM\varepsilon)^{1/p}+8b C^{2/p}\varepsilon = \widetilde {\varepsilon}.
\end{align}
Moreover, for all $\mu_2\in \cM_p^C\left(X_2\right)$, let $\mu_1=f'_{\sharp}\mu_2$ we will prove that $W_p^{a,b}\left(\mu_2, f_{\sharp}\mu_1\right)\leq \widetilde{\varepsilon}$. Since $\vert \mu_2\vert = \vert f_{\sharp}\left(f'_{\sharp}\mu_2\right)\vert$ we have $W_p^{a,b}\left(\mu_2, f_{\sharp}\mu_1\right) \leq bW_p\left( \mu_2, f_{\sharp}\left(f'_{\sharp}\mu_2\right) \right)$. Applying \cite[Lemma 2.8]{NPC} again we obtain 
\begin{align*}
W_p^p\left(\mu_2, f_{\sharp}\left(f'_{\sharp}\mu_2\right)\right) &\leq \vert\mu_2\vert \displaystyle\int_{X_2} d_2^p\left(x_2, f\left(f'\left(x_2\right)\right)\right)d\mu_2(x_2) \\
& \leq  \vert \mu_2\vert ^2 (4\varepsilon)^p\\
& \leq C^2(4\varepsilon)^p.
\end{align*}
Therefore 
\begin{align}\label{3.2}
W_p^{a,b} \left(\mu_2, f_{\sharp}\mu_1\right) \leq 4b C^{2/p}\varepsilon\leq \widetilde{\varepsilon}.
\end{align}
Combining (\ref{3.1}) and (\ref{3.2}) we have $f_{\sharp}$ is an $\widetilde{\varepsilon}$-Gromov-Hausdorff approximation.
\end{proof}
\begin{proof}[Proof of Theorem \ref{T-convergence of generalized Wasserstein spaces}] 
Since $\left\{\left(X_n, d_n\right)\right\}$ converges in the Gromov-Hausdorff topology to $(X,d)$, there exists a sequence of $\varepsilon_n$-approximations $f_n:X_n\rightarrow X$ with $\lim_{n\rightarrow \infty}\varepsilon_n=0$. By \cite[Lemma 4.1]{NPC}, there is a sequence of functions $f_n^*$ that is measurable and $5\varepsilon_n$-Gromov-Hausdorff approximations from $X_n$ to $X$. Using lemma \ref{L-convergence of generalized Wasserstein spaces} we get the result.\end{proof}

\begin{lemma}\label{L-actions on generalized Wasserstein spaces}
Let $\alpha_1,\alpha_2$ be actions of a topological group $G$ on bounded, Polish metric spaces $(X_1,d_1)$, $(X_2,d_2)$, respectively and let $C>0$. If $f: X_1\rightarrow X_2$ is an $\varepsilon$-measurable GH approximation from $\alpha_1$ to $\alpha_2$ then for every $p\geq 1$, the map $f_{\sharp}: \left(\cM_p^C(X_1), W_p^{a,b}\right)\rightarrow \left(\cM_p^C(X_2), W_p^{a,b}\right)$ is an $\widetilde{\varepsilon}$-measurable GH approximation from $(\alpha_1)_{\sharp}$ to $(\alpha_2)_{\sharp}$ where 
$$\widetilde{\varepsilon} = 8bC^{2/p}\varepsilon +b\left(9pC\left(\mathrm{diam}(X_1)^{p-1} + \mathrm{diam}(X_2)^{p-1}\right)\varepsilon\right)^{1/p}.$$
\end{lemma}
\begin{proof}
By lemma \ref{L-convergence of generalized Wasserstein spaces} we get that $f_{\sharp}$ is an $\widetilde{\varepsilon}$-Gromov-Hausdorff approximation. Therefore, to finish the proof, we only need to check that $$d_{sup} \left(f_{\sharp}\circ\left(\alpha_1\right)_{\sharp,g}, \left(\alpha_2\right)_{\sharp,g}\circ f_{\sharp}\right)\leq \widetilde{\varepsilon}.$$
Let $\mu_1\in \cM_p^C (X_1), g\in G$. Since $d_{\sup}\left(f\circ \alpha_{1,g},\alpha_{2,g}\circ f\right)\leq \varepsilon$ and $\left\vert \left(f\circ\alpha_{1,g}\right)_{\sharp}\mu_1\right\vert =\left\vert \left(\alpha_{2,g}\circ f\right)_{\sharp}\mu_1\right\vert =\left\vert \mu_1\right\vert$, applying \cite[Lemma 2.8]{NPC} we obtain \begin{align*}
W^p_p \left(\left(f\circ\alpha_{1,g}\right)_{\sharp}\mu_1,\left(\alpha_{2,g}\circ f\right)_{\sharp}\mu_1\right)&\leq \left\vert \mu_1\right\vert\displaystyle\int_{X_1} d_1^p \left(\left(f\circ\alpha_{1,g}\right)\left(x_1\right),\left(\alpha_{2,g}\circ f\right)\left(x_1\right)\right)d\mu_1(x_1)\\
&\leq C^2\varepsilon ^p.  
\end{align*}
As $\left\vert \left(f\circ\alpha_{1,g}\right)_{\sharp}\mu_1\right\vert =\left\vert \left(\alpha_{2,g}\circ f\right)_{\sharp}\mu_1\right\vert$ we have \begin{align*}
W_p^{a,b}\left(\left(f_{\sharp}\circ\left(\alpha_1\right)_{\sharp,g}\right)\left(\mu _1\right), \left(\left(\alpha_2\right)_{\sharp,g}\circ f_{\sharp}\right)\left(\mu _1\right)\right)& = W^{a,b}_p \left(\left(f\circ\alpha_{1,g}\right)_{\sharp}\mu_1,\left(\alpha_{2,g}\circ f\right)_{\sharp}\mu_1\right)\\
& \leq bW_p \left(\left(f\circ\alpha_{1,g}\right)_{\sharp}\mu_1,\left(\alpha_{2,g}\circ f\right)_{\sharp}\mu_1\right)\\ & \leq bC^{2/p}\varepsilon\\ & \leq \widetilde{\varepsilon}. 
\end{align*} \end{proof}
\begin{proof}[Proof of Theorem \ref{T-actions on generalized Wasserstein spaces}]
This theorem follows from lemma \ref{L-actions on generalized Wasserstein spaces}.\end{proof}
\begin{remark}
From \cite[Lemma 4.2]{NPC}, we see that if $\alpha_n$, $\alpha$ are isometric actions then the conclusion of theorem \ref{T-actions on generalized Wasserstein spaces} is also true for $d_{GH}$ instead of $d_{mGH}$.
\end{remark}
\section{The quotient maps of generalized Wasserstein spaces}
Let $X$ be a locally compact space. A continuous map $f:X\to X$ is proper if $f^{-1}(K)$ is compact for every compact $K\subset X$. A continuous action of a locally compact group $G$ to the right on $X$ is proper if the map $\alpha:G\times X\to X\times X$, defined by $(g,x)\mapsto (xg,x)$ is proper. Let a locally compact group $G$ act continuously and properly to the right on $X$. For $x\in X$ the orbit $G(x)$ is defined by $G(x)=\{ y\in X\; \vert \; \exists g\in G: y=xg\}$. We denote $X/G$ the \textit{orbit space} with the relation $\sim$ is defined by $x\sim y$ iff $\exists g\in G: y=xg$.
As the action is continuous and proper, the orbit space $X/G$ is Hausdorff and locally compact \cite[Chapter III, $\mathsection$4.2, Proposition 3 and $\mathsection$4.5, Proposition 9]{Bourbaki}. Therefore we can apply Riesz representation theory for Borel measures on $X/G$. 

Let $\lambda$ be a left Haar measure on $G$ and $p:X\to X/G$ be the natural quotient map. Let $f\in C_c(X)$ and $x\in X$. As the action is proper, the function $G\to \C$, $g\mapsto f(xg)$ is in $C_c(G)$. Then we can define the map $f^1:X\to \C$ by $f^1(x):=\int_Gf(xg)d\lambda(g)$ for every $x\in X$. Since $\lambda$ is left invariant we get that $f^1(xh)=f^1(x)$ for every $x\in X$ and $h\in G$. Therefore we can define the map $f^*:X/G\to \C$ by $f^*(p(x))=\int_Gf(xg)d\lambda(g)$, for every $x\in X$. It is not difficult to see that the function $f^1$ is continuous and hence $f^*$ is continuous on $X/G$ as the map $p$ is an open map. As $\supp(f)\subset Y$ for some compact subset $Y$ of $X$, one has $\supp(f^*)\subset p(Y)$, a compact subset of $X/G$ and hence $f^*\in C_c(X/G)$. As a consequence, we can define a linear function $\Phi:C_c(X)\to C_c(X/G)$ by $\Phi(f)=f^*$ and $\Phi(f)\geq 0$ for every $f\geq 0$. Applying Riesz representation theorem we get that for a Borel measure $\nu$ on $X/G$ there is a Borel measure $\widehat{\nu}$ on $X$ such that $\widehat{\nu}(f)=\nu(f^*)$ for every $f\in C_c(X)$.

 On the other hand, for every $x\in X,h\in G$ one has
$$(R_{h^{-1}}f)^1(x)=\int_GR_{h^{-1}}f(xg)d\lambda(g)=\Delta(h)\int_G f(xg)d\lambda(g)=\Delta(h)f^1(x),$$
where $R_{h}f(x)=f(xh)$ for every $x\in X,h\in G$, and $\Delta:G\to (0,\infty)$ is the right-hand modular function of $G$, i.e. $\int_Gs(gh^{-1})d\lambda(g)=\Delta(h)\int_G s(g)d\lambda(g)$ for every $h\in G$, $s\in L^1(G,\lambda)$.
 
Therefore for every $f\in C_c(X), h\in G$ one has $(R_hf)^*=\Delta(h^{-1})f^*$ and hence $$\widehat{\nu}(R_{h^{-1}}f)=\Delta(h)\widehat{\nu}(f).$$
Furthermore, we have the following result.

\begin{lemma}(\cite[Chapter 7, $\mathsection$2, Proposition 4]{Bourbaki1} or \cite[Theorem 7.3.3]{Wijsman}) Let a locally compact group $G$ act properly on the right of a locally compact space $X$ and let $\lambda$ be a left Haar measure of $G$. Then
\begin{enumerate}
\item Given a Borel measure $\nu$ on $X/G$ there exists a unique Borel measure $\widehat{\nu}$ on $X$ such that $\widehat{\nu}(f)=\nu(f^*)$ for every $f\in C_c(X)$. In addition to, one has $\widehat{\nu}(R_{h^{-1}}f)=\Delta(h)\widehat{\nu}(f),$ for every $h\in G$. 
\item Conversely, let $\mu$ be a Borel measure on $X$ such that $\mu(R_{h^{-1}}f)=\Delta(h)\mu(f),$ for every $f\in C_c(X),h\in G$. Then there exists a unique Borel measure $\mu^b$ on $X/G$ such that $\mu=\widehat{\mu^b}$.
\end{enumerate}
\end{lemma} 
The measure $\mu^b$ in the previous lemma is called the quotient of $\mu$ and $\lambda$ and is denoted by $\mu/\lambda$.

 If the acting group $G$ is unimodular, i.e. $\Delta(h)=1$ for every $h\in G$, then we get that for every Borel measure $\nu$ of $X/G$ there exists a unique $G$-invariant measure $\widehat{\nu}$ on $X$ such that $\widehat{\nu}(f)=\nu(f^*)$ for every $f\in C_c(X)$. Furthermore, if $G$ is compact 
and $\lambda$ is the normalized Haar measure of $G$ then for every $G$-invariant measure $\mu$ of $G$, the quotient measure $\mu/\lambda$ coincides with the push forward measure $p_\sharp\mu$ where $p:X\to X/G$ is the natural quotient map \cite[Proposition 7.3.5]{Wijsman}. Therefore in this case the map $p_\sharp:\cM^G(X)\to \cM(X/G)$ is bijective, where $\cM^G(X)$ is the space of all $\mu\in \cM(X)$ which is $G$-invariant. A measure $\mu\in \cM(X)$ is $G$-invariant if $\mu(Ag)=\mu(A)$ for every $g\in G$ and Borel subset $A$ of $X$. 

From now on, the acting group $G$ is compact and let $G$ act on the right of a locally compact complete separable metric space $(X,d)$ by isometries.
An element of the quotient $X/G$ will be denoted by $x^*=p(x)$. On $X/G$ we define the followings distance
$$d_{X/G} (x^*, y^*):= \inf\limits_{g\in G} d_X (gx,y)=\inf\limits_{g,h\in G} d_X(gx,hy).$$
Since $G$ is compact, for every $x,y\in X$, there exists $g\in G$ such that $
d_{X/G} (x^*,y^*)  = d_X(gx,y).$

As $(X,d)$ is a locally compact complete separable metric space, so is $(X/G,d_{X/G})$.

For every $p\geq 1$, we denote $\mathcal{M}^G_p(X)=\cM_p(X)\cap \cM^G(X)$. Let $\lambda$ be the normalized Haar measure of $G$. For every $x\in X$, the measure $\lambda_x:=\int_G \delta_{xg} d\lambda (g)$ is the unique $G$-invariant probability measure satisfying $p_\sharp\lambda_x=\delta_{x^*}$. As $\lambda_x=\lambda_y$ whenever there is some $g\in G$ such that $x=yg$, the map from $X/G$ to $\mathcal{P}(X)$ defined by $x^*\mapsto \lambda_x$ is well defined and measurable. Therefore, for every $\mu^*\in \mathcal{P}(X/G)$ we can define a $G$-invariant measure $\widehat{\mu^*}$ on $X$ by $\widehat{\mu^*}:=\int_{X/G}\lambda_xd\mu^*(x^*)$. 

Now we are ready to prove theorem \ref{T-isometry}.

\noindent\textit{Proof of Theorem \ref{T-isometry}.}

(1) We need to check that if $\mu\in \mathcal{M}_p(X)$ then $p_{\sharp}\mu\in \mathcal{M}_p(X/G)$. Since $\mu\in \mathcal{M}_p (X)$ there exists $x_0\in X$ such that $\int_X d_X^p (x_0,x)d\mu(x)<\infty$.
So 
\begin{align*}
\int_{X/G} d_{X/G}^p \left(x_0^*, x^*\right) d\left(p_{\sharp}\mu\right) (x^*) = \int_X d^p_{X/G} \left(x^*_0, p(x)\right)d \mu(x) \leq \int_Xd_X^p (x_0,x)d\mu(x) <\infty.
\end{align*}
Therefore $p_{\sharp}\mu\in\mathcal{M}_p(X/G)$.

Now we prove that $p_\sharp:\cM_p(X)\to \cM_p(X/G)$ is onto. Let $\nu^*\in \cM_p(X/G)$. Then there exists $x_0^*\in X/G$ such that 
$\int_{X/G}d^*(x^*,x_0^*)^pd\nu^*(x^*)<\infty$. Applying \cite[Theorem 3.2]{Garcia} we get that $W_p(\frac{\widehat{\nu^*}}{|\widehat{\nu^*}|},\lambda_{x_0})=W_p(\frac{\nu^*}{|\nu^*|},\delta_{x_0^*})$ and therefore
\begin{eqnarray*}
\big(\int_X d^p(x,x_0)d\widehat{\nu^*}(x)\big)^{1/p}&=&|\widehat{\nu^*}|^{1/p}W_p(\frac{\widehat{\nu^*}}{|\widehat{\nu^*}|},\delta_{x_0})\\
&\leq &|\widehat{\nu^*}|^{1/p}(W_p(\frac{\widehat{\nu^*}}{|\widehat{\nu^*}|},\lambda_{x_0})+W_p(\lambda_{x_0},\delta_{x_0}))\\
&= &|\widehat{\nu^*}|^{1/p}(W_p(\frac{\nu^*}{|\nu^*|},\delta_{x_0^*})+W_p(\int_G\delta_{x_0g}d\lambda(g),\delta_{x_0}))\\
&\leq &\big(\int_{X/G}d^*(x^*,x_0^*)^pd\nu^*(x^*)\big)^{1/p}+|\widehat{\nu^*}|^{1/p}\sup_{g\in G}d(x_0g,x_0)<\infty.
\end{eqnarray*}
Hence $\widehat{\nu^*}\in \cM_p(X)$. As we also have $\widehat{\nu^*}\in \cM^G(X)$ and $p_\sharp\widehat{\nu^*}=\nu^*$ we get that the map $p_\sharp:\cM_p(X)\to \cM_p(X/G)$ is onto.

(2) Let $\mu,\nu\in \mathcal{M}(X)$ and $\left(\gamma_1,\gamma_2\right)\in \cM_p(X)\times \cM_p(X)$  be an optimal for $W_p^{a,b}\left(\mu,\nu\right)$ such that $\vert\gamma_1\vert=\vert\gamma_2\vert$ and $\gamma_1\leq \mu, \gamma_2\leq \nu$. Let $\pi_1$ be an optimal transference between $\gamma_1$ and $\gamma_2$. Define $\pi_2=\left(p\times p\right)_{\sharp}\pi_1$ then $\pi_2\in \Pi\left(p_{\sharp}\gamma_1,p_{\sharp}\gamma_2\right)$. So that 
\begin{align*}
W_p^p\left(p_{\sharp}\gamma_1,p_{\sharp}\gamma_2\right)&\leq \vert p_{\sharp}\gamma_1\vert\int_{(X/G)\times(X/G)} d^p_{X/G}\left(x^*,y^*\right)d\pi_2\left(x^*,y^*\right) \\
&=\vert \gamma_1\vert \int_{X\times X}d^p_{X/G}\left(p(x),p(y)\right)d\pi_1 (x,y)\\
&\leq \vert \gamma_1\vert\int_{X\times X}d^p_X (x,y) d\pi_1 (x,y)\\
&=W_p^p\left( \gamma_1,\gamma_2\right).
\end{align*}
Since $\gamma_1\leq \mu,\gamma_2\leq \nu$ and $\vert\gamma_1\vert=\vert\gamma_2\vert$ we get $ p_{\sharp}\gamma_1\leq p_{\sharp}\mu, p_{\sharp}\gamma_2\leq p_{\sharp}\nu$ and $\vert p_{\sharp}\gamma_1\vert=\vert p_{\sharp}\gamma_2\vert$. On the other hand, we also have $p_\sharp \gamma_i\in \cM_p(X/G), i=1,2$. Therefore,
\begin{align*}
W_p^{a,b}\left(p_{\sharp}\mu,p_{\sharp}\nu\right)&\leq a\vert p_{\sharp}\mu- p_{\sharp}\gamma_1\vert + a\vert  p_{\sharp}\nu- p_{\sharp}\gamma_2\vert+bW_p\left( p_{\sharp}\gamma_1, p_{\sharp}\gamma_2\right)\\
& \leq a\vert \mu-\gamma_1\vert+a\vert \nu-\gamma_2\vert +bW_p\left(\gamma_1,\gamma_2\right)\\
& = W_p^{a,b}\left(\mu,\nu\right).
\end{align*}
(3) As $p_\sharp:\cM^G(X)\to \cM(X/G)$ is onto and from (2), it is sufficient to prove that $$W_p^{a,b}(p_\sharp \mu,p_\sharp\nu)\geq W_p^{a,b}(\mu,\nu) \mbox{ for every } \mu,\nu\in \cM^G(X).$$
Let $\mu,\nu\in \cM^G(X)$ and put $\mu^*=p_{\sharp}\mu,\nu^*=p_{\sharp}\nu\in \cM(X/G)$. For every $x\in X$, we define a $G$-invariant measure $\lambda_x\in\mathcal{P}(X)$ by $\lambda_x:=\int_G \delta_{xg} d\lambda (g).$
We set $\mu_0^*:=\mu^*/\vert\mu^*\vert, \nu_0^*:=\nu^*/\vert\nu^*\vert$ and define measures $\widehat{\mu_0^*},\widehat{\nu_0^*}\in \mathcal{P}(X)$ as follows
\begin{align}\label{3.3}
\widehat{\mu_0^*}=\int_{X/G} \lambda_x d\mu_0^* (x^*)\text{ and } \widehat{\nu_0^*}=\int_{X/G}\lambda_x d\nu_0^* (x^*).
\end{align}
Then $\widehat{\mu_0^*},\widehat{\nu_0^*}$ are the $G$-invariant probability measures on $X$ such that $p_\sharp\widehat{\mu_0^*}=\mu_0^*$ and $p_\sharp\widehat{\nu_0^*}=\nu_0^*$. Moreover, if we put $\mu_1=\mu/\vert\mu\vert=\mu/\vert\mu^*\vert$ and $\nu_1=\nu/\vert\nu\vert=\nu/\vert\nu^*\vert$ then $\mu_1,\nu_1$ are also the $G$-invariant probability measures on $X$ satisfying $p_{\sharp}\mu_1=\mu_0^*,p_{\sharp}\nu_1=\nu_0^*$. Therefore, $\widehat{\mu_0^*}=\mu_1,\widehat{\nu_0^*}=\nu_1$, and hence
$\mu=\vert\mu^*\vert\widehat{\mu_0^*} \text{ and } \nu=\vert\nu^*\vert \widehat{\nu_0^*}.$

Thus, we only need to prove that $W_p^{a,b}(\mu^*,\nu^*)\geq W_p^{a,b}\left(\vert \mu^*\vert\widehat{\mu_0^*},\vert\nu^*\vert\widehat{\nu_0^*}\right)$. We choose an optimal $(\widetilde{\mu}^*,\widetilde{\nu}^*)\in \cM_p(X/G)\times \cM_p(X/G)$ for $W_p^{a,b}(\mu^*,\nu^*)$ such that $\vert \widetilde{\mu}^*\vert = \vert \widetilde{\nu}^*\vert, \widetilde{\mu}^*\leq \mu^*, \widetilde{\nu}^*\leq \nu^*$ then 
\begin{align*}
W_p^{a,b}\left(\mu^*,\nu^*\right) =a\vert\mu^*-\widetilde{\mu}^*\vert+a\vert\nu^*-\widetilde{\nu}^*\vert+bW_p\left(\widetilde{\mu}^*,\widetilde{\nu}^*\right).
\end{align*}
Putting $\widetilde{\mu}_0^*:=\widetilde{\mu}^*/\vert \widetilde{\mu}^*\vert,\widetilde{\nu}_0^*:=\widetilde{\nu}^*/\vert \widetilde{\nu}^*\vert$ and we define 
\begin{align}\label{3.4}
\widetilde{\mu}_*=\int_{X/G}\lambda_xd\widetilde{\mu}_0^*(x^*) \text{ and } \widetilde{\nu}_*=\int_{X/G}\lambda_xd\widetilde{\nu}_0^*(x^*).
\end{align}
Then $\widetilde{\mu}_*, \widetilde{\nu}_*$ are the $G$-invariant probability measures on $X$ such that $p_{\sharp}\widetilde{\mu}_*=\widetilde{\mu}_0^*, p_{\sharp}\widetilde{\nu}_*=\widetilde{\nu}_0^*$. Therefore, using \cite[Theorem 3.2]{Garcia} we get that 
\begin{align*}
W_p\left(\widetilde{\mu}_*, \widetilde{\nu}_*\right)=W_p\left(\widetilde{\mu}_0^*, \widetilde{\nu}_0^*\right).
\end{align*}
So 
$W_p\left(\widetilde{\mu}^*,\widetilde{\nu}^*\right)=\vert\widetilde{\mu}^*\vert^{1/p}W_p\left(\widetilde{\mu}_0^*,\widetilde{\nu}_0^*\right)=\vert\widetilde{\mu}^*\vert^{1/p}W_p\left(\widetilde{\mu}_*, \widetilde{\nu}_*\right)=W_p\left(\vert\widetilde{\mu}^*\vert\widetilde{\mu}_*, \vert\widetilde{\nu}^*\vert\widetilde{\nu}_*\right).$

Moreover, since (\ref{3.3}) and (\ref{3.4}), for every Borel subset $A$ of $X$ one has 
\begin{align*}
\vert \widetilde{\mu}^*\vert\widetilde{\mu}_* (A)=\vert \widetilde{\mu}^*\vert \int_{X/G} \lambda_x (A)d\widetilde{\mu}_0^* (x^*)=\int_{X/G} \lambda_x(A) d\vert\widetilde{\mu}^*\vert \widetilde{\mu}_0^* (x^*)\leq \int_{X/G} \lambda_x(A) d\vert \mu^*\vert\mu_0^* (x^*)=\vert \mu^*\vert\widehat{\mu_0^*} (A).
\end{align*}
So $\vert\widetilde{\mu}^*\vert \widetilde{\mu}_*\leq \vert\mu^*\vert\widehat{\mu_0^*}$. As $\widetilde{\mu}^*\in \cM_p(X/G)$ we get that so is $\widetilde{\mu}^*_0$ and therefore similar to the proof of (1) we have $\widetilde{\mu}_*\in \cM_p(X)$. Hence $\vert\widetilde{\mu}^*\vert \widetilde{\mu}_*\in \cM_p(X)$. 
Similarly, $\vert \widetilde{\nu}^*\vert\widetilde{\nu}_*\leq \vert \nu^*\vert \widehat{\nu_0^*}$ and $\vert \widetilde{\nu}^*\vert\widetilde{\nu}_*\in \cM_p(X)$. On the other hand, $\vert \mu^*-\widetilde{\mu}^*\vert =\big\vert\vert\mu^*\vert\mu_0^*-\vert\widetilde{\mu}^*\vert\widetilde{\mu}_0^*\big\vert=\left\vert\vert\mu^*\vert\widehat{\mu_0^*}-\vert\widetilde{\mu}^*\vert\widetilde{\mu}_*\right\vert$ and $\vert \nu^*-\widetilde{\nu}^*\vert =\big\vert\vert\nu^*\vert\nu_0^*-\vert\widetilde{\nu}^*\vert\widetilde{\nu}_0^*\big\vert=\left\vert\vert\nu^*\vert\widehat{\nu_0^*}-\vert\widetilde{\nu}^*\vert\widetilde{\nu}_*\right\vert$. It implies that 
\begin{align*}
W_p^{a,b}\left(\mu^*,\nu^*\right)=a\left\vert\vert\mu^*\vert\widehat{\mu_0^*} -\vert\widetilde{\mu}^*\vert\widetilde{\mu}_*\right\vert+a\left\vert\vert\nu^*\vert\widehat{\nu_0^*} -\vert\widetilde{\nu^*}\vert\widetilde{\nu}_*\right\vert + bW_p\left(\vert \widetilde{\mu^*}\vert \widetilde{\mu}_*,\vert \widetilde{\nu^*}\vert \widetilde{\nu}_*\right).
\end{align*}
Therefore we obtain 
\begin{align*}
W_p^{a,b}\left(\mu^*,\nu^*\right)\geq W_p^{a,b}\left(\vert \mu^*\vert \widehat{\mu_0^*},\vert\nu^*\vert\widehat{\nu_0^*}\right)=W_p^{a,b}\left(\widehat{\mu^*},\widehat{\nu^*}\right).
\end{align*}
(4) follows from (1) and (3).
\dpcm
\section{The dual formulation for the $W^{a,b}_1$ distance}
In this section, we will study a dual formulation for the generalized Wasserstein $W^{a,b}_1$ distance and its consequences. Before proving theorem \ref{T-duality of generalized Wasserstein spaces}, let us recall some preparation results. 

Let $F: [0,\infty)\to (0,\infty)$ be a convex and lower semicontinuous
function. We define function $F^\circ:\R\to [-\infty,\infty]$ by $F^\circ(\varphi):=\inf_{s\geq 0}\big(\varphi s+F(s)\big)$ for every $\varphi\in \R$. As $F$ is convex the map $x\mapsto \frac{F(x)-F(0)}{x-0}$ is increasing in $(0,\infty)$ and hence we define $F_\infty':=\lim_{s\to\infty}\frac{F(s)}{s}=\sup_{s>0}\frac{F(s)-F(0)}{s}$. Now we define the functional $\cF:\cM(X)\times \cM(X)\to [0,\infty]$ by 
$$\cF(\gamma|\mu):=\int_XF(f)d\mu+F'_\infty \gamma^\perp(X),$$
where $\gamma=f\mu+\gamma^\perp$ is the Lebesgue decomposition of $\gamma$ with respect to $\mu$. 
\begin{theorem}\label{T-duality of entropy functional}
(\cite[Theorem 2.7 and Remark 2.8]{Liero})
Let $F: [0,\infty)\to (0,\infty)$ be a convex and lower semicontinuous function and $X$ be a Polish metric space. Then for every $\gamma,\mu\in \cM(X)$, we have $$\cF(\gamma|\mu)=\sup\big\{\int_XF^\circ(\varphi)d\mu-\int_X\varphi d\gamma:\varphi,F^\circ(\varphi)\in C_b(X)\big\}.$$ 
\end{theorem}
We now give the proof of the easy part of theorem \ref{T-duality of generalized Wasserstein spaces}.  
\begin{lemma}
\label{L-easy part of duality}
Suppose $X$ is a Polish metric space. For every $\mu_1,\mu_2\in\mathcal{M}(X)$, we have 
\begin{align*}
W_1^{a,b}(\mu_1,\mu_2)\geq\sup\limits_{(\varphi_1,\varphi_2)\in\Phi_W} \sum\limits_i \int_X I\left(\varphi_i(x)\right) d\mu_i(x).
\end{align*}
\end{lemma}
\begin{proof}
Let $\mu_1,\mu_2\in \cM(X)$. Let $\left(\gamma_1,\gamma_2\right)$ be an optimal for $W^{a,b}_1(\mu_1,\mu_2)$ such that $\vert \gamma_1\vert =\vert \gamma_2\vert$ and $\gamma_i\leq \mu_i,i=1,2$. Then $W^{a,b}_1\left(\mu_1,\mu_2\right)=a\left\vert \mu_1-\gamma_1\right\vert+a\left\vert \mu_2-\gamma_2\right\vert+b\,W_1\left(\gamma_1,\gamma_2\right).$

Let $\pi$ be an optimal transference between $\gamma_1$ and $\gamma_2$. We define $\gamma :=\vert \gamma_1\vert\pi$ then $\gamma_1,\gamma_2$ are the marginals of $\gamma$ and $W_1\left(\gamma_1,\gamma_2\right)=\int_{X\times X}d(x,y)d\gamma(x,y) .$

For each $i\in\{1,2\}$, since $\gamma_i\leq \mu_i$, by Radon-Nikodym theorem we get that there exists a measurable function $f_i:X\rightarrow [0,1]$ such that $\gamma_i=f_i\mu_i$. From this, we have $$a\left\vert \mu_i-\gamma_i\right\vert = \int_{X}a\left(1-f_i(x)\right)d\mu_i(x).$$
Now, for every $\left(\varphi_1,\varphi_2\right)\in\Phi_W$ we get that \begin{align*}
W^{a,b}_1\left(\mu_1,\mu_2\right) & = \sum_i\int_{X}a\left(1-f_i(x)\right)d\mu_i(x)+b\int_{X\times X} d(x,y)d\gamma(x,y)\\
& \geq \sum_i\int_{X}a\left(1-f_i(x)\right)d\mu_i(x) + \int_{X\times X} \left(\varphi_1(x)+\varphi_2(y)\right)d\gamma(x,y)\\
& = \sum_i\int_{X}\left(a\left(1-f_i(x)\right)+f_i(x)\varphi_i(x)\right)d\mu_i(x).
\end{align*}
Furthermore, as $0\leq f_i(x)\leq 1$ for all $x\in X$, we have $$I(\varphi_i(x))=\inf_{s\geq 0}\left(s\varphi_i(x)+a\vert1-s\vert\right)\leq f_i(x)\varphi_i(x)+a\left(1-f_i(x)\right),\;i=1,2.$$
Therefore, $W^{a,b}_1\left(\mu_1,\mu_2\right)\geq  \sum\limits_i \int_X I\left(\varphi_i(x)\right) d\mu_i(x),
\text{ for all }\left(\varphi_1,\varphi_2\right)\in\Phi_W.$
\end{proof}
\begin{lemma}
\label{L-duality for generalized Wassertein spaces}
If $X$ is a Polish metric space then for every $\mu_1,\mu_2\in \mathcal{M}(X)$ we have \begin{align*}
W^{a,b}_1\left(\mu_1,\mu_2\right) = \inf_{\gamma \in M}\left\lbrace a\left\vert \mu_1-\gamma_1\right\vert+a\left\vert \mu_2-\gamma_2\right\vert+b\int_{X\times X}d(x,y)d\gamma(x,y)\right\rbrace,
\end{align*}
where $\gamma_1,\gamma_2$ are the marginals of $\gamma$ and $M=\left\lbrace\gamma\in \mathcal{M}(X\times X)\mid \displaystyle\int_{X\times X}d(x,y)d\gamma(x,y)<\infty\right\rbrace$.
\end{lemma}
\begin{proof}
For any $\gamma\in M$, let $\gamma_1$ and $\gamma_2$ be the marginals of $\gamma$. Then $\vert \gamma_1\vert=\vert \gamma_2\vert$ and $\gamma_i\in \mathcal{M}_1(X), i=1,2$. Therefore \begin{align*}
a\left\vert \mu_1-\gamma_1\right\vert+a\left\vert \mu_2-\gamma_2\right\vert+b\int_{X\times X} d(x,y)d\gamma(x,y) & \geq a\left\vert \mu_1-\gamma_1\right\vert+a\left\vert \mu_2-\gamma_2\right\vert+b\,W_1\left(\gamma_1,\gamma_2\right)\\
& \geq W^{a,b}_1\left(\mu_1,\mu_2\right).
\end{align*}
So $
\inf_{\gamma \in M}\left\{a\left\vert \mu_1-\gamma_1\right\vert+a\left\vert \mu_2-\gamma_2\right\vert+b\int_{X\times X}d(x,y)d\gamma(x,y)\right\}\geq W^{a,b}_1\left(\mu_1,\mu_2\right).$

Conversely, let $\left(\widetilde{\mu}_1,\widetilde{\mu}_2\right)\in \cM_1(X)\times \cM_1(X)$ be an optimal for $W^{a,b}_1\left(\mu_1,\mu_2\right)$ and let $\widetilde{\pi}$ be an optimal transference plan between $\widetilde{\mu}_1$ and $\widetilde{\mu}_2$. Then we get that
$$W^{a,b}_1\left(\mu_1,\mu_2\right)=a\left\vert\mu_1-\widetilde{\mu}_1\right\vert+a\left\vert\mu_2-\widetilde{\mu}_2\right\vert+b\vert \widetilde{\mu}_1\vert\int_{X\times X}d(x,y)d\widetilde{\pi}(x,y).$$
We now define $\widetilde{\gamma} := \vert \widetilde{\mu}_1\vert\widetilde{\pi}$ then $\widetilde{\mu}_1,\widetilde{\mu}_2$ are the marginals of $\widetilde{\gamma}$ and $$\vert \widetilde{\mu}_1\vert\int_{X\times X}d(x,y)d\widetilde{\pi}(x,y)=\int_{X\times X}d(x,y)d\widetilde{\gamma}(x,y).$$
Thus, \begin{align*}
W^{a,b}_1\left(\mu_1,\mu_2\right)&=a\left\vert\mu_1-\widetilde{\mu}_1\right\vert+a\left\vert\mu_2-\widetilde{\mu}_2\right\vert+b\int_{X\times X}d(x,y)d\widetilde{\gamma}(x,y)\notag\\
& \geq \inf_{\gamma \in M}\left\{a\left\vert \mu_1-\gamma_1\right\vert+a\left\vert \mu_2-\gamma_2\right\vert+b\int_{X\times X}d(x,y)d\gamma(x,y)\right\}.
\end{align*}
\end{proof}
\begin{remark}
\label{R-duality for generalized Wassertein spaces}
From Proposition \ref{P-less measures for general Wassertein spaces} we also have 
\begin{align*}
W^{a,b}_1\left(\mu_1,\mu_2\right) = \inf_{\gamma \in M^\leq(\mu_1,\mu_2)}\left\lbrace a\left\vert \mu_1-\gamma_1\right\vert+a\left\vert \mu_2-\gamma_2\right\vert+b\int_{X\times X}d(x,y)d\gamma(x,y)\right\rbrace,
\end{align*}
where $M^\leq(\mu_1,\mu_2)=\left\lbrace\gamma\in M: \gamma_i\leq \mu_i, i=1,2\right\rbrace$.
\end{remark}
\begin{proof}[Proof of Theorem  \ref{T-duality of generalized Wasserstein spaces}]

We will prove this theorem in two steps. In the first step, we consider $X$ is compact. We will prove for a general Polish metric space $X$ in step 2.

\textit{Step 1. $X$ is a compact metric space.}\\ 
For any $\mu_1,\mu_2\in \mathcal{M}(X)$, using lemma \ref{L-duality for generalized Wassertein spaces} we obtain
\begin{align}\label{3.5}
W^{a,b}_1\left(\mu_1,\mu_2\right) = \inf_{\gamma \in M}\left\{a\left\vert \mu_1-\gamma_1\right\vert+a\left\vert \mu_2-\gamma_2\right\vert+b\int_{X\times X}d(x,y)d\gamma(x,y)\right\}.
\end{align}
where $\gamma_1,\gamma_2$ are the marginals of $\gamma$.\\
For each $i\in\{1,2\}$, let $\gamma_i=f_i\mu_i+\gamma_i^\perp $ be the Lebesgue decomposition of $\gamma_i$ with respect to $\mu_i$. Then we have \begin{align*}
\left\vert \mu_i-\gamma_i\right\vert & = \left\vert \left(f_i-1\right)\mu_i+\gamma_i^\perp\right\vert = \int_{X}\left\vert 1-f_i\right\vert d\mu_i+\gamma_i^\perp(X).
\end{align*}
Applying theorem \ref{T-duality of entropy functional} with $F(s)=a\vert 1-s\vert, F^o(\varphi)=I(\varphi)$ and $\cF(\gamma_i\vert \mu_i)=a\left\vert \mu_i-\gamma_i\right\vert$ we get that 
\begin{align*}
a\left\vert \mu_i-\gamma_i\right\vert = \sup \left\{\int_X I\left(\varphi_i(x)\right)d\mu_i(x)-\int_X\varphi_i(x)d\gamma_i(x) \mid\varphi_i, I\left(\varphi_i\right)\in C_b(X)\right\}. 
\end{align*}
Observe that, for every $\varphi\in \mathbb{R}$ we have \begin{align*}
I(\varphi)=\inf_{s\geq 0}\left(s\varphi+a\vert 1-s\vert\right)= \left\lbrace\begin{array}{cl}
a &\text{ if } \varphi>a\\
\varphi &\text{ if } -a\leq \varphi\leq a\\
-\infty &\text{ otherwise }
\end{array}\right..
\end{align*}
It implies that \begin{align}\label{3.6}
a\left\vert \mu_i-\gamma_i\right\vert = \sup \left\{\int_X I\left(\varphi_i(x)\right)d\mu_i(x)-\int_X\varphi_i(x)d\gamma_i(x) \mid\varphi_i\in C_b(X)\text{ and } \varphi_i(x)\geq -a,\;\forall x\in X\right\}.
\end{align}
Since (\ref{3.5}) and (\ref{3.6}) we obtain
\begin{align*}
W_1^{a,b}(\mu_1,\mu_2)=\inf\limits_{\gamma\in M}\sup\limits_{(\varphi_1,\varphi_2)\in\Phi} \left\lbrace\sum\limits_i \int_X I (\varphi_i(x))d\mu_i(x)+\int_{X\times X} \left(b.d(x,y)-\varphi_1(x)-\varphi_2(y)\right)d\gamma(x,y)\right\rbrace,
\end{align*}
where $\Phi:= \left\lbrace\left\lbrace\varphi_1,\varphi_2\right)\in C_b(X)\times C_b(X)\mid \varphi_i(x)\geq -a,\;\forall x\in X, i=1,2\right\rbrace$. We define the function $L:M\times \Phi\to \R$ by
\begin{align*}
L\left(\gamma,\left(\varphi_1,\varphi_2\right)\right)=\sum\limits_i \int_X I (\varphi_i(x))d\mu_i(x)+\int_{X\times X} \left(b.d(x,y)-\varphi_1(x)-\varphi_2(y)\right)d\gamma(x,y),
\end{align*}
for every $\gamma\in M$ and $\left(\varphi_1,\varphi_2\right)\in\Phi$. Recall that a function $g:\Phi\to \R$ is concave if $g(tx+(1-t)y)\geq tg(x)+(1-t)g(y)$ for every $x,y\in \Phi,t\in [0,1]$.
It is clear that $L\left(\cdot,\left(\varphi_1,\varphi_2\right)\right)$ is convex, and $L(\gamma,\cdot)$ is concave as $I\left(\varphi_i\right)$ is concave. Observe that $\varphi_i\in C_b(X)$ and using \cite[Lemma 4.3]{V09} we obtain  $L\left(\cdot,\left(\varphi_1,\varphi_2\right)\right)$ is lower semicontinuous in $M$ endowed with the weak*- topology. 

Next, we will estimate $\inf\limits_{\gamma\in M}L\left(\gamma,\left(\varphi_1,\varphi_2\right)\right)$ for every $(\varphi_1,\varphi_2)\in \Phi$.

(1) If $\varphi_1(x)+\varphi_2(y)\leq b.d(x,y)$ for all $x,y\in X$ then $$\inf_{\gamma \in M}\int_{X\times X}\left(b.d(x,y)-\varphi_1(x)-\varphi_2(y)\right)d\gamma(x,y) \geq 0.$$
Furthermore, if let $\gamma$ be the null measure, i.e. $\gamma (A)=0$ for every Borel subset $A$ of $X$, then $\int_{X\times X}\left(b.d(x,y)-\varphi_1(x)-\varphi_2(y)\right)d\gamma(x,y) = 0.$
Therefore, $$\inf\limits_{\gamma\in M}L\left(\gamma,\left(\varphi_1,\varphi_2\right)\right) = \sum\limits_i \int_X I (\varphi_i(x))d\mu_i(x).$$

(2) If there exist $x_0,y_0\in X$ such that $\varphi_1\left(x_0\right)+\varphi_2\left(y_0\right)>b.d\left(x_0,y_0\right)$ then we choose $\gamma = \lambda\delta_{\left(x_0,y_0\right)}$ for $\lambda >0$. Then $$L\left(\gamma,\left(\varphi_1,\varphi_2\right)\right)=\sum\limits_i \int_X I (\varphi_i(x))d\mu_i(x)+\lambda\left(b.d\left(x_0,y_0\right)-\varphi_1\left(x_0\right)-\varphi_2\left(y_0\right)\right).$$
Let $\lambda \rightarrow \infty$ we get $
\inf\limits_{\gamma\in M}L\left(\gamma,\left(\varphi_1,\varphi_2\right)\right) = -\infty.$

Therefore,
\begin{align*}
\inf\limits_{\gamma\in M}L\left(\gamma,\left(\varphi_1,\varphi_2\right)\right) = \left\{\begin{array}{cl}
\sum\limits_i \int_X I (\varphi_i)d\mu_i & \text{ if }\varphi_1(x)+\varphi_2(y)\leq b.d(x,y),\;\forall x,y\in X\\
-\infty &\text{ otherwise }
\end{array}\right..
\end{align*}
Hence
$\sup\limits_{\left(\varphi_1,\varphi_2\right)\in \Phi}\inf \limits_{\gamma \in M}L(\gamma, \varphi)=\sup\limits_{\left(\varphi_1,\varphi_2\right)\in \Phi_W}\sum \limits_{i}\int_X I\left(\varphi_i(x)\right)d\mu_i(x).$
So we only need to check that
\begin{align*}
\inf\limits_{\gamma\in M}\sup\limits_{(\varphi_1,\varphi_2)\in \Phi} L(\gamma, \varphi)=\sup\limits_{(\varphi_1,\varphi_2)\in \Phi} \inf\limits_{\gamma\in M}L(\gamma, \varphi).
\end{align*}
As we always have $\inf\limits_{\gamma\in M}\sup\limits_{(\varphi_1,\varphi_2)\in \Phi} L(\gamma, \varphi)\geq\sup\limits_{(\varphi_1,\varphi_2)\in \Phi} \inf\limits_{\gamma\in M}L(\gamma, \varphi)$, we only need to check for the case $\sup\limits_{(\varphi_1,\varphi_2)\in \Phi} \inf\limits_{\gamma\in M}L(\gamma, \varphi)$ is finite.
We choose $C>\sup\limits_{(\varphi_1,\varphi_2)\in \Phi} \inf\limits_{\gamma\in M}L(\gamma, \varphi)$ and the constant function  $\widetilde{\varphi}=\left(-a/2,-a/2\right)\in \Phi$. Then we get that 
\begin{align*}
L\left(\gamma,\widetilde{\varphi}\right)& = I(-a/2)\left(\mu_1(X)+\mu_2(X)\right) +b\int_{X\times X} d(x,y)d\gamma(x,y) + a.\gamma (X\times X)\\
& = -\dfrac{a}{2}\left(\mu_1(X)+\mu_2(X)\right) +b\int_{X\times X} d(x,y)d\gamma(x,y) + a.\gamma (X\times X).
\end{align*}
Hence the set 
$P:= \left\lbrace\gamma\in M: L\left(\gamma,\widetilde{\varphi}\right)\leq C\right\rbrace $ is bounded in the sense that there exists $K>0$ such that $\gamma(X\times X)\leq K$ for every $\gamma\in P$.
As $X$ is compact, the set $P$ is compact in the weak*-topology. Therefore, using \cite[Theorem 2.4]{Liero} we obtain 
\begin{align*}
\inf\limits_{\gamma\in M}\sup\limits_{\left(\varphi_1,\varphi_2\right)\in \Phi} L\left(\gamma, \left(\varphi_1,\varphi_2\right)\right)=\sup\limits_{\left(\varphi_1,\varphi_2\right)\in \Phi}\inf\limits_{\gamma\in M} L\left(\gamma, \left(\varphi_1,\varphi_2\right)\right).
\end{align*}
It implies that \begin{align}\label{F-step 1}
\inf_{\gamma \in M}\left\{a\left\vert \mu_1-\gamma_1\right\vert+a\left\vert \mu_2-\gamma_2\right\vert+b\int_{X\times X}d(x,y)d\gamma(x,y)\right\} = \sup\limits_{\left(\varphi_1,\varphi_2\right)\in \Phi_W}\sum \limits_{i}\int_X I\left(\varphi_i(x)\right)d\mu_i(x).
\end{align}
Hence, the proof of step 1 is completed.\\\\
\textit{Step 2. $X$ is a Polish metric space.}\\
Since $\mu_1,\mu_2\in \cM(X)$, for every $\varepsilon>0$, there exist a compact set $X_0\subset X$ such that $$\mu_i\left(X\backslash X_0\right)\leq \varepsilon,\,i=1,2.$$
We define $\mu_i^*:=\mu_{i\mid X_0},\,i=1,2$, i.e. for all Borel subset $A$ of $X$, $\mu_i^*(A)=\mu_i\left(A\cap X_0\right)$. We choose an optimal $\left(\overline{u}_1,\overline{u}_2\right)$ for $W^{a,b}_1\left(\mu_1^*,\mu_2^*\right)$ such that $\vert \overline{u}_1\vert=\vert \overline{u}_2\vert$ and $\overline{u}_i\leq \mu_i^*,\,i=1,2$. Then $$W^{a,b}_1\left(\mu_1^*,\mu_2^*\right)=a\left\vert \mu_1^*-\overline{u}_1\right\vert+a\left\vert \mu_2^*-\overline{u}_2\right\vert+b\,W_1\left(\overline{u}_1,\overline{u}_2\right).$$
Moreover, for each $i\in\{1,2\}$ we have \begin{align*}
\left\vert \mu_i^*-\overline{u}_i\right\vert & = \mu_i^*(X)-\overline{u}_i(X)\\
& = \mu_i\left(X_0\right)-\overline{u}_i(X)\\
& = \mu_i\left(X\right)-\overline{u}_i(X)-\mu_i\left(X\backslash X_0\right)\\
& \geq \left\vert \mu_i-\overline{u}_i\right\vert-\varepsilon.
\end{align*}
Therefore, \begin{align*}
W^{a,b}_1\left(\mu_1^*,\mu_2^*\right)& \geq a\left\vert \mu_1-\overline{u}_1\right\vert+a\left\vert \mu_2-\overline{u}_2\right\vert+b\,W_1\left(\overline{u}_1,\overline{u}_2\right)-2a\varepsilon\notag\\
& \geq W^{a,b}_1\left(\mu_1,\mu_2\right) - 2a\varepsilon . 
\end{align*}
On the other hand, using lemma \ref{L-duality for generalized Wassertein spaces} we get that $$W^{a,b}_1\left(\mu_1^*,\mu_2^*\right) = \inf_{\gamma \in M}\left\lbrace a\left\vert \mu_1^*-\gamma_1\right\vert+a\left\vert \mu_2^*-\gamma_2\right\vert+b\int_{X\times X}d(x,y)d\gamma(x,y)\right\rbrace,$$
where $\gamma_1,\gamma_2$ are the marginals of $\gamma$ and $M=\left\lbrace\gamma\in \mathcal{M}(X\times X)\mid \displaystyle\int_{X\times X}d(x,y)d\gamma(x,y)<\infty\right\rbrace$. Let $M^*=\left\lbrace \gamma\in M\mid \gamma \left((X\times X)\backslash \left(X_0\times X_0\right)\right)=0\right\rbrace$ then $M^*\subset M$ and thus, 

\begin{align*}
W^{a,b}_1\left(\mu_1^*,\mu_2^*\right) & \leq \inf_{\gamma \in M^*}\left\lbrace a\left\vert \mu_1^*-\gamma_1\right\vert+a\left\vert \mu_2^*-\gamma_2\right\vert+b\int_{X\times X}d(x,y)d\gamma(x,y)\right\rbrace\notag\\
& = \inf_{\gamma \in M^*}\left\lbrace a\left\vert\mu_1^*-\gamma_1\right\vert\left(X_0\right)+a\left\vert \mu_2^*-\gamma_2\right\vert\left(X_0\right)+b\int_{X_0\times X_0}d(x,y)d\gamma(x,y)\right\rbrace .
\end{align*} 
Since $X_0$ is compact, using the identity (\ref{F-step 1}) in step 1, we obtain
 \begin{align*}
\inf_{\gamma \in M^*}\left\lbrace a\left\vert\mu_1^*-\gamma_1\right\vert\left(X_0\right)+a\left\vert \mu_2^*-\gamma_2\right\vert\left(X_0\right)+b\int_{X_0\times X_0}d(x,y)d\gamma(x,y)\right\rbrace = \sup\limits_{\left(\varphi_1,\varphi_2\right)\in \Phi_W^*}\sum \limits_{i}\int_{X_0} I\left(\varphi_i\right)d\mu_i^*,
\end{align*}
where $$\Phi_W^* = \left\lbrace \left(\varphi_1,\varphi_2\right)\in C_b\left(X_0\right)\times C_b\left(X_0\right)\mid \varphi_1(x)+\varphi_2(y)\leq b.d(x,y)\text{ and }\varphi_1(x),\varphi_2(y)\geq -a,\,\forall x,y\in X_0\right\rbrace.$$
In addition, there exists $\left(\varphi_1^*,\varphi_2^*\right)\in\Phi_W^*$ such that\begin{align*}
\sum \limits_{i}\int_{X_0} I\left(\varphi_i^*(x)\right)d\mu_i^*(x)\geq \sup\limits_{\left(\varphi_1,\varphi_2\right)\in \Phi_W^*}\sum \limits_{i}\int_{X_0} I\left(\varphi_i(x)\right)d\mu_i^*(x)-\varepsilon.
\end{align*}
Then we get that \begin{align}\label{F-inequality of duality}
\sum \limits_{i}\int_{X_0} I\left(\varphi_i^*(x)\right)d\mu_i^*(x)\geq W^{a,b}_1\left(\mu_1,\mu_2\right) - (2a+1)\varepsilon.
\end{align}
Next, for each $x\in X$ we define $\overline{\varphi}_1(x):=\min \left\lbrace \inf_{y\in X_0}\left(b.d(x,y)-\varphi_2^*(y)\right),a\right\rbrace .$ 
As the function $X\to \R, x\mapsto \inf_{y\in X_0}\left(b.d(x,y)-\varphi_2^*(y)\right)$ is Lipschitz, we have $\overline{\varphi}_1\in C(X)$. For each $x\in X_0$, we have $\varphi_1^*(x)+\varphi_2^*(y)\leq b.d(x,y),\,\forall y\in X_0.$
So $$\varphi_1^*(x)\leq \inf_{y\in X_0}\left(b.d(x,y)-\varphi_2^*(y)\right),\,\forall x\in X_0.$$
Moreover, we also have $\varphi_1^*(x)+\varphi_2^*(x)\leq b.d(x,x)=0,\,\forall x\in X_0.$
And from $\varphi_i^*(x)\geq -a,\,\forall x\in X_0, i=1,2$, we get that $$\varphi_i^*(x)\in [-a,a],\,\forall x\in X_0,i=1,2.$$
Therefore, $\overline{\varphi}_1(x)\geq \varphi_1^*(x)$ for every $x\in X_0$. Besides that, for any $x\in X$ one has $$b.d(x,y)-\varphi_2^*(y)\geq -\varphi_2^*(y)\geq -a,\,\forall y\in X_0.$$
Thus, $\overline{\varphi}_1(x)\in [-a,a]$ for every $x\in X$.\\
Now, we define, for each $y\in X$, $\overline{\varphi}_2(y):= \inf_{x\in X}\left(b.d(x,y)-\overline{\varphi}_1(x)\right).$
Then $\overline{\varphi}_2\in C(X)$ and $$\overline{\varphi}_1(x)+ \overline{\varphi}_2(y)\leq b.d(x,y),\,\forall x,y\in X.$$
By the same arguments as above, we still have $\overline{\varphi}_2(y)\in [-a,a],\,\forall y\in X$ and $\overline{\varphi}_2\geq \varphi_2^*$ on $X_0$. Therefore $(\overline{\varphi}_1,\overline{\varphi}_2)\in \Phi_W$.

Since the function $I$ is nondecreasing, applying (\ref{F-inequality of duality}), we obtain \begin{align*}
\sum \limits_{i}\int_{X} I\left(\overline{\varphi}_i(x)\right)d\mu_i(x) &=\sum \limits_{i}\int_{X\backslash X_0} I\left(\overline{\varphi}_i(x)\right)d\mu_i(x)+ \int_{X_0}I\left(\overline{\varphi}_i(x)\right)d\mu_i^*(x)\\
& \geq -a\left(\mu_1\left(X\backslash X_0\right)+\mu_2\left(X\backslash X_0\right)\right)+\sum \limits_{i}\int_{X_0}I\left(\varphi_i^*(x)\right)d\mu_i^*(x)\\
& \geq -2a\varepsilon + W^{a,b}_1\left(\mu_1,\mu_2\right)-(2a+1)\varepsilon\\
& = W^{a,b}_1\left(\mu_1,\mu_2\right) -(4a+1)\varepsilon .
\end{align*}
Therefore, $\sup\limits_{\left(\varphi_1,\varphi_2\right)\in \Phi_W}\sum \limits_{i}\int_X I\left(\varphi_i(x)\right)d\mu_i(x)\geq W^{a,b}_1\left(\mu_1,\mu_2\right).$

Applying lemma \ref{L-duality for generalized Wassertein spaces}, we get that 
\begin{eqnarray*}
W^{a,b}_1\left(\mu_1,\mu_2\right) &=& \sup\limits_{\left(\varphi_1,\varphi_2\right)\in \Phi_W}\sum \limits_{i}\int_X I\left(\varphi_i(x)\right)d\mu_i(x)\\
&=&\sup\limits_{(\varphi_1,\varphi_2)\in\Phi_W} \sum\limits_i \int_X \inf\limits_{s\geq 0} \left(s\varphi_i(x)+a\vert 1-s\vert\right) d\mu_i(x).
\end{eqnarray*}
\end{proof}
\begin{remark}
In the case $X$ is compact, theorem \ref{T-duality of generalized Wasserstein spaces} is a special case of \cite[Theorem 4.11]{Liero} although its  statement there is slightly different from ours as they consider lower semicontinuous functions $\varphi_1,\varphi_2$. For the completeness, we present a proof for this compact case in step 1 and it follows the ideas of the proof of \cite[Theorem 4.11]{Liero}. 
\end{remark}
Let $(X,d)$ be a metric space. For a function $f:X\to \R$, we denote 
$$\|f\|_{Lip}:=\sup_{x,y\in X,x\neq y}\frac{|f(x)-f(y)|}{d(x,y)}.$$
Now, using the techniques of the proof of \cite[Theorem 1.14]{V03} and applying theorem \ref{T-duality of generalized Wasserstein spaces} we are ready to prove theorem \ref{T-flat metrics}.
\begin{proof}[Proof of Theorem \ref{T-flat metrics}]
For every $(\psi,\varphi)\in \Phi_W$, we define $\varphi^d(x):=\inf_{y\in X}[b.d(x,y)-\varphi(y)]$ for every $x\in X$. Then $\varphi^d$ is $b$-Lipschitz function and $\varphi^d(x)\in [-a,a]$ for every $x\in X$. Therefore $\varphi^d\in \bF$. Now we define $\varphi^{dd}(y):=\inf_{x\in X}[b.d(x,y)-\varphi^d(x)]$ for every $y\in X$. Then $\varphi^{dd}$ is $b$-Lipschitz and 
$$\varphi^d(x)+\varphi^{dd}(y)\leq b.d(x,y), \mbox{ for every } x,y\in X.$$
As $-a\leq \varphi^{d}(x)\leq a$ we also get that $-a\leq \varphi^{dd}(y)\leq a$ for every $y\in X$. Therefore we have $\varphi^{dd}\in \bF$ and $(\varphi^d,\varphi^{dd})\in \Phi_W$.

On the other hand, as $\psi(x)+\varphi(y)\leq b.d(x,y)$ for every $x,y\in X$ we get that 
$$\psi(x)\leq \inf_{y\in X}[b.d(x,y)-\varphi(y)]=\varphi^d(x) \mbox{ for every } x\in X.$$
Similarly, from the definitions of $\varphi^{dd}$ we also have $\varphi^{dd}(y)\geq \varphi(y)$ for every $y\in Y$. Hence 
\begin{align*}
\int_{X} I\left(\psi\right)d\mu+\int_{X} I\left(\varphi\right)d\nu \leq \int_{X} I\left(\varphi^d\right)d\mu+\int_{X} I\left(\varphi^{dd}\right)d\nu.
\end{align*}
Therefore, $$\sup_{(\psi,\varphi)\in \Phi_W}\big\{\int_{X} I\left(\psi\right)d\mu+\int_{X} I\left(\varphi\right)d\nu \big\}\leq \sup_{\varphi\in C_b(X)}\big\{\int_{X} I\left(\varphi^d\right)d\mu+\int_{X} I\left(\varphi^{dd}\right)d\nu \big\}.$$
As $\varphi^d$ is $b$-Lipschitz we get $$-\varphi^d(x)\leq \inf_{y\in X}[b.d(x,y)-\varphi^d(y)].$$
On the other hand, $\inf_{y\in X}[b.d(x,y)-\varphi^d(y)]\leq -\varphi^d(x)$. Hence $$\varphi^{dd}(x)=\inf_{y\in X}[b.d(x,y)-\varphi^d(y)]=-\varphi^d(x).$$
Thus
\begin{eqnarray*}
\sup_{(\psi,\varphi)\in \Phi_W}\big\{\int_{X} I\left(\psi\right)d\mu+\int_{X} I\left(\varphi\right)d\nu \big\}&\leq& \sup_{\varphi\in C_b(X)}\big\{\int_{X} I\left(\varphi^d\right)d\mu+\int_{X} I\left(\varphi^{dd}\right)d\nu \big\}\\
&= &\sup_{\varphi\in C_b(X)}\big\{\int_{X} I\left(\varphi^d\right)d\mu+\int_{X} I\left(-\varphi^{d}\right)d\nu \big\}\\
&\leq &\sup_{\varphi\in \bF}\big\{\int_{X} I\left(\varphi\right)d\mu+\int_{X} I\left(-\varphi\right)d\nu \big\}\\
&\leq& \sup_{(\psi,\varphi)\in \Phi_W}\big\{\int_{X} I\left(\psi\right)d\mu+\int_{X} I\left(\varphi\right)d\nu \big\}.
\end{eqnarray*}
So we must have equality everywhere and get the result.
\end{proof}
\begin{remark}
1) Theorem \ref{T-flat metrics} has been proved in \cite[Theorem 2]{PR16} for the case $a=b=1$ and $X=\R^n$ by a different method.

2) (\cite{Hanin1,Hanin2}) Let $(X,d)$ be a Polish metric space. We denote by $\cM^s(X)$ the space of all signed Borel measures with finite mass on $X$. Let $\cM^0(X)$ be the set of all $\mu\in \cM^s(X)$ such that $\mu(X)=0$. For every $\mu\in \cM^0(X)$ we denote by $\Psi_\mu$ the set of all nonnegative measures $\gamma\in \cM(X\times X)$ such that $\lambda(X\times A)-\lambda(A\times X)=\mu(A)$ for every Borel $A\subset X$. Then we define for every $\mu\in \cM^0(X)$,
$$\|\mu\|^0_d:=\inf_{\gamma\in \Psi_\mu}\big\{\int_{X\times X}d(x,y)d\gamma(x,y)\big\}.$$
Now, on the vector space $\cM^s(X)$ we define an extension Kantorovich-Rubinstein norm as following
$$\|\mu\|_d:=\inf_{\nu\in \cM^0(X)}\big\{\|\nu\|^0_d+|\mu-\nu|(X)\big\}, \mbox{ for every } \mu\in \cM^s(X).$$
Then from \cite[Theorem 0]{Hanin1} (when $X$ is compact) or \cite[Theorem 1]{Hanin2} (when $X$ is a general Polish metric space), applying Hahn-Banach theorem we get that 
$$\|\mu\|_d=\sup\big\{\int_X fd(\mu-\nu):f\in \bF\big\},$$
where $\bF:=\big\{f\in C_b(X), \|f\|_\infty\leq 1, \|f\|_{Lip}\leq 1\big\}$.
We thank Benedetto Piccoli and Francesco Rossi for pointing \cite{Hanin1} out to us, and we have found \cite{Hanin2} after that.
\end{remark}
The corollary \ref{C-geodesic space} is from theorem \ref{T-flat metrics} as following.
\begin{proof}[Proof of Corollary \ref{C-geodesic space}]
It is clear that $\left(\mathcal{M}(X), W^{a,b}_1\right)$ is complete, this follows from proposition \ref{P-completeness}. For every $\mu,\nu\in\mathcal{M}(X)$, we define $\sigma :=(\mu+\nu)/2$. Then using theorem \ref{T-flat metrics} we obtain \begin{align*}
W^{a,b}_1(\mu,\sigma) & = \sup_{f\in\bF}\int_Xfd(\mu-\sigma)\\
& = \dfrac{1}{2}\sup_{f\in\bF}\int_Xfd(\mu-\nu)\\
& = \dfrac{1}{2}W^{a,b}_1(\mu,\nu).
\end{align*}
Similarly, $W^{a,b}_1(\sigma,\nu)=\frac{1}{2}W^{a,b}_1(\mu,\nu)$. Hence, applying \cite[Theorem 2.4.16]{Burago} or \cite[Lemma 2.1]{Sturm} we get the result.
\end{proof}
Using theorem \ref{T-flat metrics} we get another proof of \cite[Lemma 5]{PRT}.
\begin{corollary}
For every $\mu,\nu,\eta\in \cM(X)$ we have 
$$W_1^{a,b}(\mu+\eta,\nu+\eta)=W^{a,b}_1(\mu,\nu).$$
\end{corollary}
From theorem \ref{T-flat metrics} we also get a similar result in \cite[Lemma 1.5]{Basso}.
\begin{corollary}
Let $\left(X_1,d_1\right)$ and $\left(X_2,d_2\right)$ be two Polish metric spaces. If $\psi:\left(X_1,d_1\right)\rightarrow\left(X_2,d_2\right)$ is an isometry map then the map $\psi_\sharp:\left(\mathcal{M}\left(X_1\right), W^{a,b}_1\right)\rightarrow \left(\mathcal{M}\left(X_2\right), W^{a,b}_1\right)$ is also an isometry.
\end{corollary}
\begin{proof} For every $\mu,\nu\in \cM(X)$, it is clear that $W^{a,b}_1\left(\mu,\nu\right)\geq W^{a,b}_1\left(\psi_\sharp\mu,\psi_\sharp\nu\right)$ and $\psi_\sharp$ is surjective. Hence, we need to show that $W^{a,b}_1\left(\psi_\sharp\mu,\psi_\sharp\nu\right)\geq W^{a,b}_1\left(\mu,\nu\right)$.\\
Let $\bF_i =\left\lbrace f\in C_b(X_i), \Vert f\Vert_\infty \leq a, \Vert f\Vert_{Lip}\leq b\right\rbrace, i=1,2$. By theorem \ref{T-flat metrics} one has
\begin{align*}
W_1^{a,b} \left(\psi_{\sharp}\mu, \psi_{\sharp}\nu\right)=\sup\limits_{g\in \bF_2}\int_{X_2} g\,d\left(\psi_{\sharp}\mu- \psi_{\sharp}\nu\right) = \sup\limits_{g\in \bF_2} \int_{X_1} g\circ\psi \,d\left(\mu-\nu\right).
\end{align*}
For every $f\in \bF_1$ and every $y\in X_2$ we define $h(y):= \inf\limits_{x\in X_1} \left[b.d_2\left(y,\psi (x)\right) + f(x) \right]$. Then $h$ is $b$-Lipschitz and $h(y)\geq -a$ for every $y\in X_2$. Since $\psi$ is surjective, for every $y\in X_2$, there exists $x'\in X_1$ such that $\psi(x')=y$. Thus, $h(y)\leq f(x')\leq a$. Therefore, $h\in \bF_2$. Moreover, since $f$ is $b$-Lipschitz, for every $x\in X_1$ one has 
\begin{align*}
f(x)&=\inf\limits_{x_1\in X_1} \left[ f(x_1) + \vert f(x) -f(x_1)\vert \right]\\ &\leq \inf\limits_{x_1\in X_1} \left[ f(x_1) + b.d_1(x,x_1) \right]\\ 
&= \inf\limits_{x_1\in X_1} \left[ f(x_1) + b.d_2\left(\psi(x) ,\psi( x_1)\right) \right]\\&\leq f(x).
\end{align*}
Therefore, $f(x)=h\left(\psi (x)\right)$, for all $x\in X_1$ or $f=h\circ \psi$. Hence, we get that 
\begin{align*}
\int_{X_1} f\, d(\mu-\nu) =\int_{X_1} h\circ \psi \, d(\mu-\nu) \leq \sup\limits_{g\in \bF_2}\int_{X_1} g\circ\psi \,d\left(\mu-\nu\right)= W_1^{a,b} \left(\psi_{\sharp}\mu, \psi_{\sharp}\nu\right).
\end{align*}
So that $W_1^{a,b} \left(\mu, \nu\right)\leq W_1^{a,b} \left(\psi_{\sharp}\mu, \psi_{\sharp}\nu\right)$.
\end{proof} 
Now using theorem \ref{T-flat metrics} we give another proof of theorem \ref{T-isometry} for the case $p=1$.
\begin{corollary}
Let $(X,d)$ be a Polish metric space. Then for every $a,b>0$ the push forward map $p_\sharp:(\cM^G(X),W_1^{a,b})\to (\cM(X/G),W_1^{a,b})$ is an isometry.
\end{corollary} 
\begin{proof}
From part 1) of theorem \ref{T-isometry} we know that $W_1^{a,b}(p_\sharp\mu,p_\sharp\nu)\leq W_1^{a,b}(\mu,\nu)$ for every $\mu,\nu\in \cM(X)$.

Now we prove that for every $\mu_1,\mu_2\in \cM_1^G(X)$, we have $W_1^{a,b}(\mu_1,\mu_2)\leq W_1^{a,b}(p_\sharp\mu_1,p_\sharp\mu_2)$. Let $\mu_1,\mu_2\in \cM_1^G(X)$. Recall that 
$\bF:=\{f\in C_b(X):\|f\|_\infty\leq a, \|f\|_{Lip}\leq b\}$ and we define
$$\bF^*:=\{f\in C_b(X/G):\|f\|_\infty\leq a, \|f\|_{Lip}\leq b\}.$$ 
For every $f\in \bF$, the map $f^1:X\to \R,$ defined by $f^1(x)=\int_Gf(xg)d\lambda(g)$ is well defined and $f^1(xh)=f(x)$ for every $x\in X,h\in G$ and hence we can define the map $f^*:X/G\to \R$ by $f^*(p(x))=f^1(x)$ for every $x\in X$. It is clear that $f^*\in C_b(X)$ and $\|f^*\|_\infty\leq a$. Now we check that $\|f^*\|_{Lip}\leq b$. For every $x^*,y^*\in X/G$ with $x^*\neq y^*$ there exist $x_0\in x^*,y_0\in y^*$ such that $d(x_0,y_0)=d^*(x^*,y^*)$. As the action is isometry and $f$ is $b$-Lipschitz, for every $x,y\in X$ we have 
$
|f^1(x)-f^1(y)|\leq \int_G|f(xg)-f(yg)|d\lambda(g)\leq b.d(x,y).$
Therefore, $$\frac{|f^*(x^*)-f^*(y^*)|}{d(x^*,y^*)}=\frac{|f^1(x_0)-f^1(y_0)|}{d(x_0,y_0)}\leq b.$$
Hence $\|f^*\|_{Lip}\leq b$ and therefore $f^*\in \bF^*$. On the other hand, as $\mu_1,\mu_2\in \cM_1^G(X)$, one has
\begin{eqnarray*}
\int_{X/G}f^*(x^*)d(p_\sharp \mu_1-p_\sharp\mu_2)(x^*)&=&\int_Xf^1(x)d(\mu_1-\mu_2)(x)\\
&=&\int_X\int_Gf(xg)d\lambda(g)d(\mu_1-\mu_2)(x)\\
&=&\int_G\int_Xf(xg)d(\mu_1-\mu_2)(x)d\lambda(g)\\
&=&\int_Xf(x)d(\mu_1-\mu_2)(x).
\end{eqnarray*}
Therefore, applying theorem \ref{T-flat metrics} we get that $W_1^{a,b}(\mu_1,\mu_2)\leq W_1^{a,b}(p_\sharp\mu_1,p_\sharp\mu_2).$
\end{proof}

For every $\mu_1,\mu_2\in \cM(X)$ we denote by $\Opt^{a,b}(\mu_1,\mu_2)$ the set of all $\gamma\in \cM(X\times X)$ such that $\int_{X\times X}d(x,y)d\gamma(x,y)<\infty$, $\gamma_i\leq \mu_i$, $i=1,2$, and 
$$W_1^{a,b}(\mu_1,\mu_2)=a|\mu_1-\gamma_1|+a|\mu_2-\gamma_2|+b\int_{X\times X}d(x,y)d\gamma(x,y),$$
where $\gamma_1,\gamma_2$ are the marginals of $\gamma$.

\begin{lemma}\label{L-existence of optimal plan}
Let $(X,d)$ be a Polish metric space. For every $\mu_1,\mu_2\in \cM(X)$ the set $\Opt^{a,b}(\mu_1,\mu_2)$ is a nonempty, convex and compact subset of $\cM(X\times X)$. 
\end{lemma}
\begin{proof}
It is clear that $\Opt^{a,b}(\mu_1,\mu_2)$ is convex.

From remark \ref{R-duality for generalized Wassertein spaces}, we choose a sequence of $\gamma_n\in \cM(X\times X)$ such that $\int_{X\times X}d(x,y)d\gamma_n(x,y)<\infty$, $(\pi_i)_\sharp\gamma_n\leq \mu_i$ for every $i=1,2, n\in \N$ and 
$$a\left\vert \mu_1-(\pi_1)_\sharp\gamma_n\right\vert+a\left\vert \mu_2-(\pi_2)_\sharp\gamma_n\right\vert+b\int_{X\times X}d(x,y)d\gamma_n(x,y)\to W_1^{a,b}(\mu_1,\mu_2).$$
As $\mu_1,\mu_2\in \cM(X)$ we get that $\{(\pi_i)_\sharp\gamma_n\}_{n\in \N}$ is tight for every $i=1,2$. Therefore for every $\varepsilon>0$ there exist compact subsets $K_\varepsilon, L_\varepsilon$ of $X$ such that 
$$(\pi_1)_\sharp\gamma_n(X\setminus K_\varepsilon)<\varepsilon \mbox{ and } (\pi_2)_\sharp\gamma_n(X\setminus L_\varepsilon)<\varepsilon, \mbox{ for every } n\in \N.$$
And hence $\gamma_n(X\times X\setminus K_\varepsilon\times L_\varepsilon)\leq (\pi_1)_\sharp\gamma_n(X\setminus K_\varepsilon)+(\pi_2)_\sharp\gamma_n(X\setminus L_\varepsilon)<2\varepsilon$, for every $n\in \N$. Therefore $\{\gamma_n\}_{n\in \N}$ is tight. As $\mu_i\in \cM(X)$ and $(\pi_i)_\sharp\gamma_n\leq \mu_i$ for every $i=1,2, n\in \N$ we get that $\{\gamma_n\}_{n\in \N}$ is bounded. 
Hence applying Prokhorov's theorem, passing to a subsequence we can assume that 
$\gamma_n\to \gamma$ as $n\to \infty$ in the weak*-topology for some $\gamma\in \cM(X\times X)$.

Since the metric function $d$ is nonnegative lower semicontinuous on $X\times X$, we can write $d$ as the pointwise limit of a nondecreasing sequence of nonnegative, continuous functions  $(c_k)_{k\in\mathbb{N}}$ on $X\times X$. Replacing $c_k$ by $\min\{c_k,k\}$, we can assume that each $c_k$ is bounded. By monotone convergence, one has $\int_{X\times X}c_k(x,y)d\gamma (x,y)\to \int_{X\times X}d(x,y)d\gamma (x,y)$ as $k\to \infty$. As $c_k\leq d$ for every $k\in \mathbb{N}$, we have $\int_{X\times X}d(x,y)d\gamma_n(x,y)\geq \int_{X\times X}c_k(x,y)d\gamma_n(x,y)$. Moreover, since $c_k$ is bounded and continuous, we get that $$\liminf_{n\rightarrow \infty}\int_{X\times X}d(x,y)d\gamma_n(x,y)\geq \lim_{n\rightarrow \infty}\int_{X\times X}c_k(x,y)d\gamma_n(x,y)=\int_{X\times X}c_k(x,y)d\gamma (x,y).$$
Therefore,
$$\liminf_{n\rightarrow \infty}\int_{X\times X}d(x,y)d\gamma_n(x,y)\geq \lim_{k\rightarrow \infty}\int_{X\times X}c_k(x,y)d\gamma(x,y)=\int_{X\times X}d(x,y)d\gamma (x,y).$$
As $W_1^{a,b}(\mu_1,\mu_2)$ is finite we get that $\int_{X\times X}d(x,y)d\gamma(x,y)<\infty$. As $\gamma_n\to \gamma$ as $n\to \infty$ in the weak*-topology, applying \cite[Theorem 6.1 page 40]{Par} we also get that $$\limsup_{n\to \infty}\gamma_n(X\times X)\leq\gamma(X\times X)\leq \liminf_{n\to\infty}\gamma_n(X\times X).$$
Therefore $\gamma(X\times X)=\lim_{n\to \infty}\gamma_n(X\times X)$. 

Next we will prove that  $(\pi_i)_\sharp\gamma\leq \mu_i$ for every $i=1,2.$ Let $A$ be a Borel subset of $X$. Applying \cite[Theorem 6.1 page 40]{Par} again we get that $\gamma(U\times X)\leq \liminf_{n\to\infty}\gamma_n(U\times X)\leq \mu(U)$ for every $U\subset X$ open.
Therefore
\begin{eqnarray*}
(\pi_1)_\sharp\gamma(A)&=&\gamma(A\times X)\\ 
&=&\inf\{\gamma(W): W\subset X\times X \mbox{ open }, A\times X\subset W\}\\
&\leq& \inf\{\gamma(U\times X): U\subset X\mbox{ open }, A\subset U\}\\
&\leq& \inf\{\mu(U): U\subset X\mbox{ open }, A\subset U\}=\mu(A).
\end{eqnarray*}
This means $(\pi_1)_\sharp \gamma\leq \mu_1.$ Similarly, we get that $(\pi_2)_\sharp \gamma\leq \mu_2.$
Therefore, $a\left\vert \mu_i-(\pi_i)_\sharp\gamma_n\right\vert\to a\left\vert \mu_i-(\pi_i)_\sharp\gamma\right\vert$ as $n\to \infty$, $i=1,2$. Hence, we get that $$a\left\vert \mu_1-(\pi_1)_\sharp\gamma\right\vert+a\left\vert \mu_2-(\pi_2)_\sharp\gamma\right\vert+b\int_{X\times X}d(x,y)d\gamma (x,y)\leq W^{a,b}_1\left(\mu_1,\mu_2\right)$$ 
Therefore, $\Opt^{a,b}(\mu_1,\mu_2)$ is nonempty.

Now we prove that $\Opt^{a,b}(\mu_1,\mu_2)$ is a compact subset of $\cM(X\times X)$. Let $\{\gamma_n\}_{n\in \N}$ be a sequence in $\Opt^{a,b}(\mu_1,\mu_2)$. Using the same argument as above we can get a subsequence of $\{\gamma_n\}_{n\in \N}$ converging to some $\gamma\in \Opt^{a,b}(\mu_1,\mu_2)$ in the weak*-topology. 
\end{proof}
Next, we will provide the optimality conditions for generalized Wasserstein distances in theorem \ref{T-existence of dual optimal} and theorem \ref{T-optimal plan and optimal dual}. These results are versions of \cite[Theorem 4.14 and Theorem 4.15]{Liero} for generalized Wasserstein distances. Note that as our nonsmooth entropy function $F(s)=a|1-s|$ is not superlinear and the cost function $b.d(\cdot,\cdot)$ does not have compact sublevels when $X$ is a general Polish metric space, they do not satisfy coercive conditions as in \cite[Theorem 4.14 and Theorem 4.15]{Liero}.
\begin{theorem}\label{T-existence of dual optimal}
Let $(X,d)$ be a Polish metric space and let $a,b>0$. For every $\mu_1,\mu_2\in \mathcal{M}(X)$, there exists $\left(\varphi_1,\varphi_2\right)\in\Phi_W$ such that $$W^{a,b}_1\left(\mu_1,\mu_2\right)=\sum_i\int_X I\left(\varphi_i(x)\right)d\mu_i(x).$$
\end{theorem}
Before giving the proof of theorem \ref{T-existence of dual optimal}, we prove the following elementary lemma.
\begin{lemma}\label{L-pointwise convergent of Lipschitz functions}
Let $(X,d)$ be a separable metric space and let $a,b>0$. If a sequence $\left\{\varphi_n\right\}$ in $C_b(X)$ satisfies that $\varphi_n$ is $b$-Lipschitz  and $\left\vert\varphi_n(x)\right\vert\leq a$ for every $x\in X$ and every $n\in\mathbb{N}$ then $\left\{\varphi_n\right\}$ has a pointwise convergent subsequence on $X$.
\end{lemma}
\begin{proof}
Let $S=\left\lbrace s_1,s_2,\ldots\right\rbrace$ be a countable dense subset of $X$. As $\vert\varphi_n(s)\vert\leq a$ for every $n\in\mathbb{N},s\in S$, using a standard diagonal argument there exists a subsequence of $\left\{\varphi_n\right\}$ which we still denote by $\left\{\varphi_n\right\}$ such that $\varphi_n(s)$ converges as $n\to \infty$ for every $s\in S$.
As $S$ is dense in $X$, for every $x\in X$ and every $\varepsilon >0$ there exists $s\in S$ such that $d\left(x,s\right) < \varepsilon/b$. Since $\left\lbrace \varphi_m (s)\right\rbrace$ converges, there exists $N>0$ such that for every $m_1,m_2>N$ we have $\left\vert \varphi_{m_1}(s) - \varphi_{m_2} (s)\right\vert < \varepsilon$. Furthermore, $\varphi_{m_1}$ and $\varphi_{m_2}$ are $b$-Lipschitz on $X$. Therefore,
\begin{align*}
\left\vert \varphi_{m_1}(x) - \varphi_{m_2} (x)\right\vert &\leq \left\vert \varphi_{m_1}(x) - \varphi_{m_1} (s)\right\vert + \left\vert \varphi_{m_2}(x) - \varphi_{m_2} (s)\right\vert +\left\vert \varphi_{m_1}(s) - \varphi_{m_2} (s)\right\vert \\
&< 2b.d\left(x,s\right) +\varepsilon\\
&< 3\varepsilon.
\end{align*}
Hence, we get the result.
\end{proof}

\begin{proof}[Proof of Theorem \ref{T-existence of dual optimal}]
For every $n\in \mathbb{N},n\geq 1$, using theorem \ref{T-duality of generalized Wasserstein spaces}, we get that there exists $\left(\overline{\varphi}_{1,n},\overline{\varphi}_{2,n}\right)\in\Phi_W$ such that $$\sum_i\int_XI\left(\overline{\varphi}_{i,n}(x)\right)d\mu_i(x)\geq W^{a,b}_1\left(\mu_1,\mu_2\right)-\dfrac{1}{n}.$$
For every $x\in X$, we define $\widetilde{\varphi}_{1,n}(x):=\inf_{y\in X}\left[b.d(x,y)-\overline{\varphi}_{2,n}(y)\right]$ and for every $y\in X$ we define $\widetilde{\varphi}_{2,n}(x):=\inf_{x\in X}\left[b.d(x,y)-\widetilde{\varphi}_{1,n}(x)\right]$. Then $\left(\widetilde{\varphi}_{1,n}, \widetilde{\varphi}_{2,n} \right)\in \Phi_W$ and $\widetilde{\varphi}_{i,n}$ is $b$-Lipschitz on $X, i=1,2$. Moreover, $\widetilde{\varphi}_{i,n}(x)\geq \overline{\varphi}_{i,n}(x)$ for every $x\in X, i=1,2$. Thus, $$\sum_i\int_XI\left(\widetilde{\varphi}_{i,n}(x)\right)d\mu_i(x)\geq \sum_i\int_XI\left(\overline{\varphi}_{i,n}(x)\right)d\mu_i(x).$$
Applying lemma \ref{L-pointwise convergent of Lipschitz functions}, there exists a subsequence $\left\{\widetilde{\varphi}_{1,n_k}\right\}_k$ pointwise convergent to $\widetilde{\varphi}_1$ on $X$. Then $\widetilde{\varphi}_1(x)\in [-a,a]$ for every $x\in X$. We now consider the subsequence $\left\{\widetilde{\varphi}_{2,n_k}\right\}_k$, it is clear that $\widetilde{\varphi}_{2,n_k}$ is $b$-Lipschitz and $\widetilde{\varphi}_{2,n_k}(y)\in [-a,a]$ for every $y\in X$. Thus, using lemma \ref{L-pointwise convergent of Lipschitz functions} again we obtain that there exists a subsequence $\left\{\widetilde{\varphi}_{2,n_{k_l}}\right\}_l$ pointwise convergent to $\widetilde{\varphi}_2$ on $X$. Then $\widetilde{\varphi}_2 (y)\in [-a,a]$ for every $y\in X$. For every $x,y\in X$, from $\widetilde{\varphi}_{1,n_{k_l}}(x)+ \widetilde{\varphi}_{2,n_{k_l}}(y)\leq b.d(x,y)$ we also have $\widetilde{\varphi}_1(x)+\widetilde{\varphi}_2(y)\leq b.d(x,y)$. Since the Lebesgue's dominated convergence theorem and $I(\varphi)=\varphi$ for every $\varphi \in [-a,a]$, we obtain
\begin{align*}
\sum_i\int_XI\left(\widetilde{\varphi}_{i}(x)\right)d\mu_i(x)& = \lim_{l\rightarrow \infty} \sum_i\int_XI\left(\widetilde{\varphi}_{i,n_{k_l}}(x)\right)d\mu_i(x)\\
& \geq \lim_{l\rightarrow \infty} \sum_i\int_XI\left(\overline{\varphi}_{i,n_{k_l}}(x)\right)d\mu_i(x)\\
& \geq W^{a,b}_1\left(\mu_1,\mu_2\right).
\end{align*}
Now, for every $x\in X$, we define $\varphi_1(x):=\inf_{y\in X}\left[b.d(x,y)-\widetilde{\varphi}_2(y)\right]$ and for every $y\in X$, we define $\varphi_2(y):=\inf_{x\in X}\left[b.d(x,y)-\varphi_1(x)\right]$. Then $\left(\varphi_1,\varphi_2\right)\in \Phi_W$ and $\varphi_i(x)\geq \widetilde{\varphi}_i(x)$ for every $x\in X, i=1,2$. Hence, we get that \begin{align*}
\sum_i\int_XI\left(\varphi_{i}(x)\right)d\mu_i(x) & \geq \sum_i\int_XI\left(\widetilde{\varphi}_{i}(x)\right)d\mu_i(x)\\
& \geq W^{a,b}_1\left(\mu_1,\mu_2\right).
\end{align*}
Therefore, $\sum_i\int_XI\left(\varphi_{i}(x)\right)d\mu_i(x)=W^{a,b}_1\left(\mu_1,\mu_2\right)$.
\end{proof}
We say that a pair $\left(\varphi_1,\varphi_2\right)\in\Phi_W$ is a dual optimal for $W^{a,b}_1\left(\mu_1,\mu_2\right)$ if $W^{a,b}_1\left(\mu_1,\mu_2\right)=\sum_i\int_XI\left(\varphi_{i}(x)\right)d\mu_i(x)$.
\begin{corollary}
Let a compact group $G$ act on the right of a locally compact Polish metric space $(X,d_X)$ by isometries. Let $p:X\to X/G$ be the natural quotient map and let any $\mu_1,\mu_2\in \mathcal{M}^G(X)$. If a pair $\left(\overline{\varphi}_1,\overline{\varphi}_2\right)$ is a dual optimal for $W^{a,b}_1\left(p_\sharp\mu_1,p_\sharp\mu_2\right)$ then $\left(\varphi_1,\varphi_2\right)$ is also a dual optimal for $W^{a,b}_1\left(\mu_1,\mu_2\right)$, where $\varphi_i$ is defined by $\varphi_i:={\overline{\varphi}_{i}}\circ p,i=1,2$.
\end{corollary}
\begin{proof}
Since $\overline{\varphi}_i\left(x^*\right)\in [-a,a]$  for every $x^*\in X/G$ and $\overline{\varphi}_i\in C_b(X/G)$, we get that $\varphi_i(x)\in [-a,a]$ for every $x\in X$ and $\varphi_i\in C_b(X)$, $i=1,2$, since $p$ is continuous. Moreover, for every $x,y\in X$ one has $\varphi_1(x)+\varphi_2(y)=\overline{\varphi}_1\left(x^*\right)+\overline{\varphi}_2\left(y^*\right)\leq b.d_{X/G}\left(x^*,y^*\right)\leq b.d_{X}\left(x,y\right)$. Therefore, $\left(\varphi_1,\varphi_2\right)\in\Phi_W$. Since $\left(\overline{\varphi}_1,\overline{\varphi}_2\right)$ is a dual optimal for $W^{a,b}_1\left(p_\sharp\mu_1, p_\sharp\mu_2\right)$ and using theorem \ref{T-isometry} we get that \begin{align*}
W^{a,b}_1\left(\mu_1,\mu_2\right) & = W^{a,b}_1\left(p_\sharp\mu_1, p_\sharp\mu_2\right)\\
& = \sum_i\int_{X/G} I \left(\overline{\varphi}_i\left(x^*\right)\right)dp_\sharp\mu_i\left(x^*\right)\\
& = \sum_i\int_XI\left(\overline{\varphi}_i(p(x))\right)d\mu_i(x)\\
& = \sum_i\int_X I(\varphi_i(x))d\mu_i(x).
\end{align*}
Hence, $\left(\varphi_1,\varphi_2\right)$ is a dual optimal for $W^{a,b}_1\left(\mu_1,\mu_2\right)$.
\end{proof}
The above result has been proved for Wasserstein distances in \cite[Corollary 3.4]{Garcia}.

Next, we provide the conditions between a optimal plan $\gamma\in\Opt^{a,b}\left(\mu_1,\mu_2\right)$ and a dual optimal $\left(\varphi_1,\varphi_2\right)$.
\begin{theorem}
\label{T-optimal plan and optimal dual}
Let $(X,d)$ be a Polish metric space and let $\mu_1,\mu_2\in\mathcal{M}(X),\gamma\in M^\leq(\mu_1,\mu_2)$. Then for every $a,b>0$ the plan $\gamma\in \Opt^{a,b}\left(\mu_1,\mu_2\right)$ if and only if there exist a pair $\left(\varphi_1,\varphi_2\right)\in \Phi_W$ and two Borel subsets $A_1,A_2$ of $X$ that satisfy the following conditions 
\begin{enumerate}[(i)]
\item $\gamma_i(X\minus A_i)=\mu_i^\perp (A_i)=0,i=1,2$, where $\gamma_i$ is the marginal of $\gamma$ and $\mu_i=g_i\gamma_i+\mu_i^\perp$ is the Lebegues decomposition of $\mu_i$ with respect to $\gamma_i$.
\item $\varphi_1(x)+\varphi_2(y)=b.d(x,y)$ $\gamma$-a.e in $X\times X$.
\item $\left(a-\varphi_i(x)\right)\left(1-f_i(x)\right)=0$ $\mu_i$-a.e in $A_i$, $i=1,2$, where $f_i: X\rightarrow [0,1]$ is the Borel density of $\gamma_i$ with respect to $\mu_i$.
\item $\varphi_i(x)=a$ $\mu_i^\perp$-a.e in $X\minus A_i, i=1,2$.
\end{enumerate}
\end{theorem}
\begin{proof}
Let $\gamma\in \Opt^{a,b}\left(\mu_1,\mu_2\right)$. For each $i\in\{1,2\}$, since $\mu_i^\perp\perp \gamma_i$, there exists a Borel subset $A_i$ of $X$ such that $\gamma_i(X\minus A_i)=\mu_i^\perp (A_i)=0$. By theorem \ref{T-existence of dual optimal}, let $\left(\varphi_1,\varphi_2\right)\in \Phi_W$ be a dual optimal for $W^{a,b}_1\left(\mu_1,\mu_2\right)$. Since $\varphi_i(x)\in [-a,a]$ for every $x\in X$, $I\left(\varphi_i(x)\right)=\inf_{s\geq 0}(s\varphi_i(x)+a|1-s|)=\varphi_i(x),i=1,2$. Hence, we get that \begin{align*}
\sum_i\int_X\varphi_i(x)d\mu_i(x) & = W^{a,b}_1\left(\mu_1,\mu_2\right)\\
& = a\left(\left\vert\mu_1-\gamma_1\right\vert+\left\vert\mu_2-\gamma_2\right\vert\right)+b\int_{X\times X}d(x,y)d\gamma (x,y)\\
& \geq \sum_i\int_X a\left(1-f_i(x)\right)d\mu_i (x)+\int_X\varphi_1(x)d\gamma(x,y)+\int_X\varphi_2(y)d\gamma(x,y)\\
& = \sum_i \big[\int_{A_i} a\left(1-f_i(x)\right)d\mu_i (x)+\int_{X\minus A_i} a\left(1-f_i(x)\right)d\mu_i (x)+\int_{A_i}\varphi_i(x)d\gamma_i(x)\big]\\
& = \sum_i \big[\int_{A_i} \left[a\left(1-f_i(x)\right)+ f_i(x)\varphi_i(x)\right] d\mu_i(x)+\int_{X\minus A_i} a\left(1-f_i(x)\right)d\mu_i (x)\big].
\end{align*} 
Since $\gamma_i(X\minus A_i)=0,i=1,2$ and $\varphi_i(x)\leq a$ for every $x\in X$, one has $$\int_{X\minus A_i} a\left(1-f_i(x)\right)d\mu_i (x) = a\int_{X\minus A_i}  d\mu_i^\perp (x)-a\gamma_i(X\minus A_i)\geq \int_{X\minus A_i} \varphi_i(x)d\mu_i^\perp (x).$$
Moreover, from $f_i(x)\in [0,1], \varphi_i(x)\in [-a,a]$ for every $x\in X$, we get that $\left(a-\varphi_i(x)\right)\left(1-f_i(x)\right)\geq 0$ or $a\left(1-f_i(x)\right)+f_i(x)\varphi_i(x)\geq \varphi_i(x)$ for every $x\in X$. Therefore,
\begin{align*}
\sum_i\int_X\varphi_i(x)d\mu_i(x) & \geq \sum_i\big(\int_{A_i}\varphi_i(x)d\mu_i(x)+\int_{X\minus A_i}\varphi_i(x)d\mu_i^\perp (x)\big)\\
& = \sum_i\int_X\varphi_i(x)d\mu_i(x).
\end{align*}
Hence, we must have equality everywhere and we get the conditions $(i)-(iv)$.

Conversely, assume that there exist $\left(\varphi_1,\varphi_2\right)\in \Phi_W$ and two Borel subsets $A_1,A_2$ of $X$ that satisfy four conditions $(i)-(iv)$. Since the conditions $(i)$ and $(ii)$ we obtain \begin{align*}
a\left(\left\vert\mu_1-\gamma_1\right\vert+\left\vert\mu_2-\gamma_2\right\vert\right)+b\int_{X\times X}d(x,y)d\gamma (x,y) & = \sum_i\int_X a\left(1-f_i(x)\right)d\mu_i(x)+\int_X\varphi_i(x)d\gamma_i(x)\\
&= \sum_i\int_{A_i}\left[a\left(1-f_i(x)\right)+f_i(x)\varphi_i(x)\right]d\mu_i(x)\\
&+\int_{X\minus A_i}a\left(1-f_i(x)\right)d\mu_i(x).
\end{align*}
On the other hand, from the conditions $(i)$ and $(iv)$, for each $i\in\{1,2\}$ we get that $$\int_{X\minus A_i}a\left(1-f_i(x)\right)d\mu_i(x)=a\int_{X\minus A_i}d\mu_i^\perp(x)-a\gamma_i\left(X\minus A_i\right)=\int_{X\minus A_i}\varphi_i(x)d\mu_i^\perp(x).$$
Furthermore, the condition $(iii)$ implies that $$\int_{A_i}\left[a\left(1-f_i(x)\right)+f_i(x)\varphi_i(x)\right]d\mu_i(x)=\int_{A_i}\varphi_i(x)d\mu_i(x),\;i=1,2.$$
Hence, we get that \begin{align*}
a\left(\left\vert\mu_1-\gamma_1\right\vert+\left\vert\mu_2-\gamma_2\right\vert\right)+b\int_{X\times X}d(x,y)d\gamma (x,y) & = \sum_i\int_{A_i}\varphi_i(x)d\mu_i(x)+\int_{X\minus A_i}\varphi_i(x)d\mu_i^\perp(x)\\
&= \sum_i\int_X I\left(\varphi_i(x)\right)d\mu_i(x)\\
& \leq W^{a,b}_1\left(\mu_1,\mu_2\right).
\end{align*}
However, we always have the opposite inequality. Therefore, $\gamma\in \Opt^{a,b}(\mu_1,\mu_2)$.
\end{proof}

\end{document}